\documentclass[10pt,a4paper]{article}
\usepackage[latin,english]{babel}
\usepackage[T1]{fontenc}
\usepackage{eufrak}
\usepackage{geometry}
\usepackage{amsmath}
\usepackage{amssymb}
\usepackage{amsbsy}
\usepackage{amsthm}
\usepackage{makeidx}
\usepackage{mathtools}
\usepackage{mathabx, mathrsfs, dsfont}
\usepackage{cases}
\usepackage{braket}
\usepackage[toc,page]{appendix}
\usepackage{dirtytalk}
\usepackage{pdfsync}
\usepackage{graphicx}
\usepackage{cancel}
\graphicspath{ {./Figures/} } 
\usepackage{epstopdf}
\usepackage{enumitem}
\usepackage{todonotes}
\usepackage{fancyhdr}
\usepackage{math}
\usepackage{cite}
\usepackage{color}
\usepackage{tikz}
\usetikzlibrary{calc}
\usetikzlibrary{decorations.markings} 
\usepackage{subcaption}
\usepackage{bbm}
\usepackage{authblk}
\usepackage[colorlinks=true, pdfstartview=FitV, urlcolor=blue, citecolor=red, linkcolor=blue,pdfencoding=unicode,unicode=true,psdextra]{hyperref}
\usepackage{marginnote}
\mathtoolsset{showonlyrefs}

\DeclareMathOperator*{\essesup}{essesup}
\DeclareMathOperator*{\Ker}{Ker}
\DeclareMathOperator*{\Res}{Res}
\DeclareMathOperator*{\Ind}{Index}
\DeclareMathOperator*{\Coker}{Coker}
\usetikzlibrary{matrix,decorations.markings}
\usetikzlibrary{decorations.markings}
\usetikzlibrary{quotes,angles}

\usepackage{esint}

\usepackage{titlesec}
\titlespacing*{\subsection}{0pt}{1.2ex}{1.2ex}

\newcommand{\trace}[1]{\mathrm{Tr}\left( #1\right)}

\def\pmtwo#1#2#3#4{\left( \begin{array}{cc}#1&#2\\#3&#4\end{array}\right)}

\newtheorem{RHP}[theorem]{RHP}

\numberwithin{equation}{section}

\title{Painlev\'e Universality classes for the maximal amplitude solution of the Focusing Nonlinear Schr\"{o}dinger Equation with randomness}

\author[1]{Aikaterini Gkogkou}
\author[1]{Guido Mazzuca}
\author[1]{Kenneth D. T-R McLaughlin}
\affil[1]{Department of Mathematics, Tulane University, New Orleans, LA}

\begin{document}
	\maketitle

\begin{abstract}
We establish universality for extremal solutions of the focusing nonlinear Schr\"{o}dinger equation.  Extremal solutions are $N$-soliton solutions that achieve the theoretical maximal amplitude and diverge as $N \to \infty$.  We consider extremal solutions with the discrete eigenvalues randomly drawn from sub-exponential distributions, and identify two distinct universality classes, determined by the macroscopic structure of the spectrum: the \textit{Painlevé--III} rogue-wave solution, where the eigenvalues take the form $\lambda_j = v_j + i \mu_j$, and the \textit{Painlevé--V} rogue wave solution, where  $\lambda_j = -\zeta \, j + v_j + i \mu_j$, with $0 < \zeta < 1$. (In both cases, $\mu_{j}$ and $v_{j}$ are subexponential random variables.)  Universality can then be summarized as follows: independently of the specific distribution of the eigenvalues, the rescaled solutions converge locally to a deterministic profile governed by the Painlevé-III equation in the first regime, and the Painlevé-V equation in the second.  These results demonstrate that the formation of Painlevé-type rogue waves is a universal phenomenon robust to randomness.
\end{abstract}

\section{Introduction}

Rogue waves, often referred to as \textit{freak waves} or \textit{monster waves}, are exceptionally large and spontaneous ocean waves with their average height being at least three times the height of the surrounding waves. They are very unpredictable, quite unexpected and mysterious. As Akhmediev describes \cite{akhmediev2009waves,akhmediev2009extreme}, they are waves ``which come from nowhere and disappear with no trace.'' From a practical standpoint, their unpredictability and extreme amplitude, pose serious hazards to ships and offshore platforms, making the study of their formation and dynamics crucial.

Historically, rogue waves were dismissed as exaggerated sailor reports. Today, they are recognized as statistically significant phenomena, representing extreme, low-probability events. Their rarity makes statistical approaches and probabilistic models fundamental to understand the mechanisms underlying their formation. For instance, Dematteis et al. in \cite{dematteis2018rogue}, used large-deviation theory to identify the most likely precursor wave configurations and estimate the probability of extreme rogue waves in random deep-sea conditions.  Beyond the ocean, rogue waves are also observed in optical fibers \cite{Solli2007,Kibler2010,Kibler2012,Frisquet2013}, and in plasmas \cite{Moslem2011,Bailung2011}, highlighting their widespread presence in diverse nonlinear systems. For a comprehensive overview of rogue wave phenomena and their physical applications, we refer the reader to the survey by Onorato et al. \cite{Onorato2013}. Recently, some experimental evidence of the existence of rogue waves has emerged. For instance,  Chabchoub et al. \cite{chabchoub2011rogue} experimentally generated and reported the observation of a \textit{Peregrine breather}, the prototypical model of rogue waves, in a controlled water‑tank setup. Additionally, Dematteis et al. \cite{dematteis2019experimental}, generated random waves in a long water flume and demonstrated that rogue waves correspond to rare hydrodynamic \textit{instanton} configurations predicted by stochastic statistical models.

Water waves can be effectively described by a variety of nonlinear wave equations. In particular, the focusing nonlinear Schrödinger (NLS) equation
\begin{equation}\label{eq:NLS}
i \partial_t \psi  = -\frac{1}{2} \partial_x^2 \psi - | \psi |^2 \psi\,, \quad (x,t) \in \R^2,
\end{equation}
is a canonical model for the slow modulation of weakly nonlinear, quasi-monochromatic wave packets in deep water, see \cite{zakharov1968}. The Peregrine breather, given as follows
\begin{equation}\label{eq:perregrine}
\psi(x,t) = \left[ 1 - \frac{4(1+2 i t)}{1 + 4 x^2 + 4 t^2} \right] e^{i t}
\end{equation}
has been suggested as a possible model for rogue-wave phenomena, since it is localized in both space and time (albeit slowly decaying). A particularly notable feature of the Peregrine breather is that its peak amplitude reaches three times that of the background wave, $\psi_{\mathrm{bg}}(x,t)=e^{i t}$. Originally derived by Peregrine \cite{peregrine1983water} to model extreme water-wave events on a constant-amplitude background, it was later experimentally generated and observed for the first time in \cite{chabchoub2011rogue}. Beyond the Peregrine breather, the
Akhmediev breather \cite{akhmediev1985generation,akhmediev1986modulation}, and the Kuznetsov--Ma breather \cite{kuznetsov1977solitons,ma1979perturbed} are relatively simple exact solutions of the focusing NLS equation that are often considered models of rogue waves.

From a mathematical perspective, certain types of rogue waves have been extensively studied using a Riemann--Hilbert approach. For example, in \cite{Bilman2019a, Bilman2020}, Bilman, Ling, and Miller consider a Riemann-Hilbert problem (RHP) with $k$ simple poles, and they rescale the pole locations so that all poles coalesce to a single point, forming a pole of order $k$.  In the limit as $k \to \infty$, they establish the existence of a large-$k$ limiting profile, referred to as the \textit{rogue wave of infinite order}, which is described in terms of a member of the Painlevé--III hierarchy \cite{Bilman2020}.  Other studies have used Riemann--Hilbert analysis to generate rogue waves from high-order breather solutions \cite{Bilman2022}, or high-order soliton solutions \cite{Bilman2019}, or high-order solutions belonging to a one-parameter family which includes the previous two classes \cite{Bilman2024}.

The focusing NLS equation \eqref{eq:NLS} is an integrable system as shown by Zakharov and Shabat in \cite{zakharov1972exact}. Its Lax pair formulation is given by:
\begin{equation}
\label{eq:ZS1}
\partial_{x} \Phi = \begin{pmatrix}
-i z & \psi\\
-\wo{\psi} & iz
\end{pmatrix} \Phi, \quad
 \partial_{t} \Phi = \begin{pmatrix}
-iz^2 + \frac{i}{2} |\psi|^2 & z \psi + \frac{i}{2} \psi_{x} \\ \\[0.1pt]
-z \wo{\psi} + \frac{i}{2} \wo{\psi}_{x} & i z^2 - \frac{i}{2} |\psi|^2
\end{pmatrix} \Phi,
\end{equation}
where $z \in \mathbb{C}$ is the spectral parameter, and $\Phi = \Phi(z;x,t)$ are simultaneous solutions of \eqref{eq:ZS1}. The NLS equation admits a broad collection of fundamental solutions. The most famous being the one-soliton solution
\begin{eqnarray}
\label{eq:1sol}
 \psi_{s}(x,t) = 2 \eta \ \mbox{sech}\left(2 \eta ( x + 2 \xi t - x_{0})\right)e^{- 2i \left[ \xi x + \left( \xi^{2} - \eta^{2}\right)t\right]} e^{- i \phi_{0}} 
\end{eqnarray}
parametrized by four parameters: $\xi, \eta, \phi_{0}$, and $x_{0}$. In addition to solitons, the NLS equation supports a rich variety of other solutions, including periodic and quasi-periodic waves, dispersive shock waves, breather solutions, rogue wave solutions, and multi-soliton solutions. In this work, we focus on multi-soliton solutions, which we refer to as $N$-soliton solutions, where $N$ denotes the number of solitons. Our interest lies in the setting in which these solutions attain their maximal amplitude, corresponding to the formation of a rogue-wave profile.

Since it is integrable, the NLS equation can be solved through the inverse scattering transform (IST), which consists of three components: the direct scattering problem, the time evolution, and the inverse problem. We refer the reader to \cite{NLS_book,borghese_long_2018} for a detailed overview of the scattering and inverse scattering theory of the focusing NLS equation among other nonlinear partial differential equations. In the direct problem, one extracts the scattering data, which consists of the reflection coefficient $r(z)$, the discrete eigenvalues $\{ \lambda_n \}_{n=1}^{N}$ (each associated with one soliton), and the normalization constants $\{ c_n \}_{n=1}^N$. For the rest of the paper, we denote the scattering data by:
\[\mathcal{D}_N  = \{ r(z), \{\lambda_n, c_n \}_{n=1}^N\}.\]
The time evolution of the scattering data is straightforward: the eigenvalues do not change in time, while the reflection coefficient and normalization constants explicitly evolve as:
\begin{equation}
\label{eq:spectral_RHP}
\mathcal{D}_N(t) = \left\{ r(z,t), \{\lambda_{n}, c_{n}(t) \}_{n=1}^{N} \right\} = \left\{ r(z)e^{2 i t z^{2} }, \{\lambda_{n}, c_{n}e^{ 2 i t \lambda_{n}^{2}  } \}_{n=1}^{N} \right\}.
\end{equation}
The inverse problem can be formulated as the following RHP (see \cite{borghese_long_2018} for more details):
\begin{RHP}\label{rhp:reflection_RHP}
Given spectral data $\mathcal{D}_N(t)$, find a $2 \times 2$ matrix-valued function $M(z)$ such that:
\begin{itemize}
	\item $\ {M}(z)$ is analytic in  $ \C \setminus \R $, except for simple poles located at the points $\{ \lambda_{n}, \overline{\lambda_{n}}\}_{n=1}^{N}$.
    \item For $z \in \mathbb{R}$, $M$ has boundary values $M_{+}(z)$ and $M_{-}(z)$, which satisfy the jump relation 
        \begin{equation}
           M_+(z)=\  M_-(z) \  J_M(z), \quad  J_M = \begin{pmatrix}
                1+|r(z)|^2 & \wo{r(z)}e^{-2i\theta(z;x,t)} \\
                r(z)e^{2i\theta(z;x,t)} & 1
            \end{pmatrix},
        \end{equation}
where $\theta(z;x,t) = zx + z^2t$.
	\item At each of the poles $\lambda_{n}$ in $\mathbb{C}^{+}$, $\ {M}(z)$ satisfies the residue condition 
	\begin{equation}
    \label{eq:rescondCP}
		\Res_{\lambda_n} M(z) = \lim_{z \to \lambda_n} \left[ M(z)\begin{pmatrix}
			0 & 0 \\
			c_n e^{2i\theta(\lambda_n;x,t)} & 0
		\end{pmatrix}\right]\,,
    \end{equation}
and at each pole $\overline{\lambda_{n}}$ in $\mathbb{C}^{-}$:
\begin{equation}
\Res_{\wo {\lambda_n}} M(z) = \lim_{z \to \wo{\lambda_n}} \left[ M(z) \begin{pmatrix}
			0 & -\wo{c_n} e^{-2i\theta(\wo{\lambda_n};x,t)} \\
			0 & 0
		\end{pmatrix}\right].
	\end{equation}
\item $\ {M}(z) = \  I + \cO(z^{-1})$ as $z \to \infty$.
\end{itemize}
\end{RHP}
Since we are interested in pure solitonic (i.e., \textit{reflectionless}) solutions, we set $r(z) \equiv 0$ for all $z \in \mathbb{C}$, in which case the jump condition on $\mathbb{R}$ disappears. Therefore, RHP \ref{rhp:reflection_RHP} reduces to finding a meromorphic function with $N$ simple poles in the upper half plane and their complex conjugates in the lower half plane, and prescribed asymptotic behavior for large-$z$ given by the last condition in RHP \ref{rhp:reflection_RHP}.

While the IST and RHP analysis provide a powerful, analytic framework for constructing soliton solutions, there exists another equivalent way of obtaining the pure solitonic solution of the NLS equation: the Darboux method, also known as the \textit{dressing method}. This method iteratively generates multi-soliton solutions by “dressing” a simpler seed solution, typically the trivial zero solution. This method will be explained in Section \ref{sec:DBmethod} below.

One important application of this method is a characterization of \textit{extremal multi-soliton solutions.}  First, it is well known (see, for example, \cite{GravaJenkinsMazzucaMcLaughlin2024}) that an $N$-soliton solution satisfies the inequality
\begin{equation}
|\psi_{N}(x,t)| \le 2 \sum_{n=1}^{N}\Im\left( \lambda_{n} \right), \ \mbox{ for all } x,t.
\end{equation}
It was also shown in \cite{GravaJenkinsMazzucaMcLaughlin2024} that this upper bound is sharp.  Indeed, as we will see, for any collection of $N$ eigenvalues, there is a one-parameter family of normalization constants $\{c_{n}\}_{n=1}^{N}$ so that $\psi_{N}(0,0)$ achieves this upper bound.  We will consider sequences of these $N$-soliton extremal solutions, and show that as $N \to \infty$, new universal behaviors emerge.

As $N\to\infty$, we establish two different universality results:  the (suitably rescaled) solutions converge to a new fundamental solution satisfying a Painlevé equation independently of the distribution of the (random) discrete eigenvalues. We identify two distinct universality classes, distinguished by the structure of the discrete eigenvalues: the \textit{Painlevé--III} and \textit{Painlevé--V} rogue-wave solutions. In the Painlevé--III case, the discrete eigenvalues are
\[
\lambda_j = v_j + i \mu_j,
\]
whereas for Painlevé--V they are
\[
\lambda_j = -\zeta \, j + v_j + i \mu_j, \qquad 0 < \zeta < 1.
\]
In both cases, $v_j$ and $\mu_j$ are sub-exponential random variables. Universality can then be summarized as follows: regardless of the specific realizations of the amplitudes and velocities, provided they are sub-exponential random variables and with normalization constants chosen to maximize the \(N\)-soliton solution, the resulting maximal peak always corresponds to either a Painlevé--III or Painlevé--V rogue-wave profile, see Figure \ref{fig:universality}.

The paper is organized as follows. Section \ref{sec:res} presents the main results. The analysis establishing these results is developed in Sections \ref{sec:PIII} and \ref{sec:PV}. We defer the more technical proofs to the Appendices.

\begin{figure}
    \centering
    \includegraphics[width=0.46\linewidth]{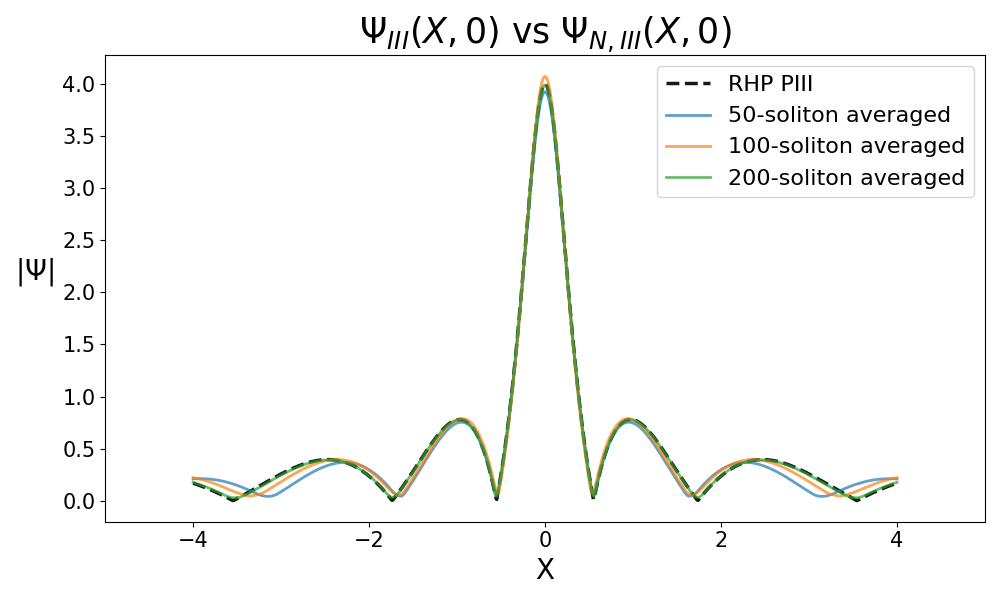}
     \includegraphics[width=0.46\linewidth]{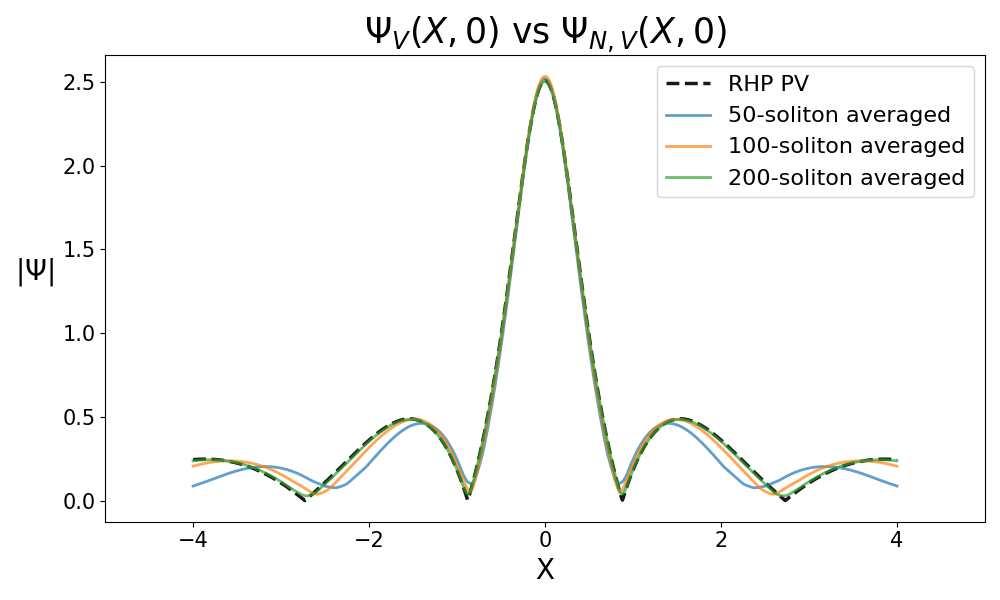}
     \includegraphics[width=0.46\linewidth]{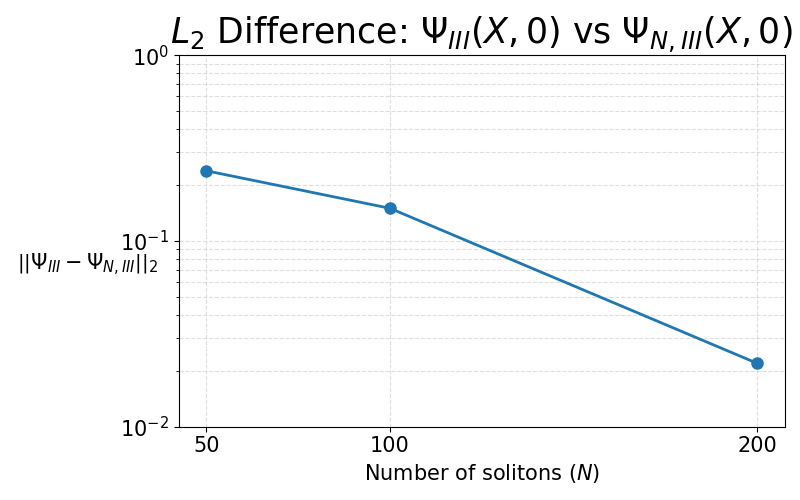}
     \includegraphics[width=0.46\linewidth]{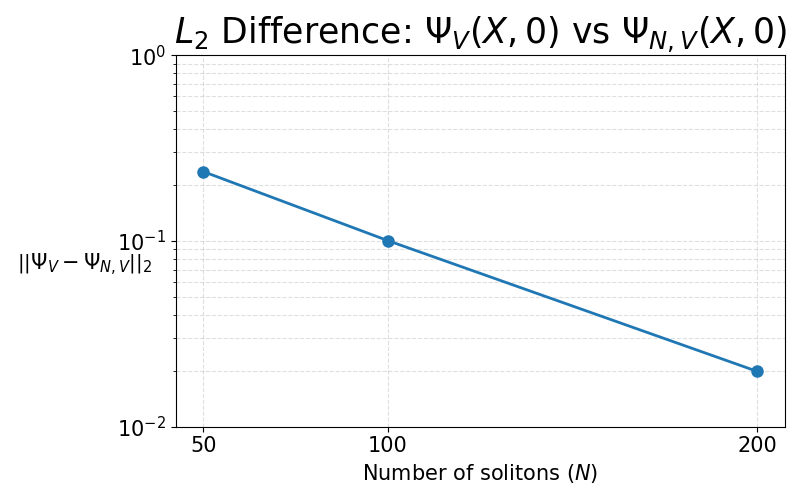}
    \caption{\textbf{Top Panels.} Painlev\'e III and V rogue waves vs random maximal soliton solution ($50-100-200$ solitons). The poles are chosen as $\lambda_j = v_j + i\mu_j$ and $\lambda_j = -0.3j+v_j + i\mu_j $, where $v_j$ are i.i.d. Gaussian random variables $\cN(0,15)$ and $\mu_j$ are i.i.d. Chi-squared distribution of parameter $4$ ($\chi^2(4)$) in the left panel, and of parameter $2$  in the right one ($\chi^2(2)$). We averaged over $10$ realizations. \textbf{Bottom Panels.} Semi-log plot of $\norm{\Psi_{III}(X,0)-\Psi_{N,III}(X,0)}_2$, and $\norm{\Psi_{V}(X,0)-\Psi_{N,V}(X,0)}_2$ for $N=50,\,100,\,200$, same initial data as before.}
    \label{fig:universality}
\end{figure}

\section{Statement of the results}\label{sec:res}

In this section, we precisely state our results. First, we recall the definition of \textit{sub-exponential} random variables.

\begin{definition}[See  \cite{Vershynin2018} Proposition 2.6]
\label{def:subexp}
    Let $\fX$ be a random variable, we say that $\fX$ is a sub-exponential random variable if there exists a $K>0$ such that
    \begin{equation}
        \mathbb{P}(|\fX|>t)\leq 2\exp\left(-\frac{t}{K}\right)\,,\qquad \forall\, t\geq0\,.
    \end{equation}
\end{definition}
A classical example of sub-exponential random variable is the $\chi^2(\alpha)$-distribution, $\alpha\in\R$, which is the random variable with probability density function
\begin{equation}
f(x;\alpha) = \begin{cases}
    \frac{x^{\alpha/2-1}e^{-\frac{x}{2}}}{2^{\alpha/2}\Gamma\left(\frac{\alpha}{2}\right)} \qquad &x\geq 0\\ \\
    0 \qquad & \text{otherwise}
\end{cases}
\end{equation}
where $\Gamma(x)$ is the Gamma-function \cite[Ch. 5]{DLMF}. 
We now introduce the scattering data and state the first set of assumptions used throughout this work.
     \begin{assumption}[Painlev\'e-III Rogue Wave] \label{hp:random_scattering_PIII}
     For each fixed $N>0$, we generate the scattering data 
\begin{equation}
    \mathcal{D}_{N,III} = \left\{r \equiv 0,
\left\{  
\lambda_{n}, c_{n}
\right\}_{n=1}^{N}
    \right\}
\end{equation}
as follows. Let the poles $\lambda_n=v_n+i\mu_n$, and in our probabilistic calculations, we assume the following:
\begin{enumerate}

\item  The amplitudes $\mu_n$ are independent identically distributed (i.i.d.) random variables 
\begin{equation}
\mu_n \ \sim \ \mathcal D_1,
\end{equation}
where $\mathcal D_1$ is sub-exponential random variable, with positive support and mean $\mu_{\mathcal D_1}$;
\item  The velocities $v_n$ are independent identically distributed (i.i.d.) random variables 
\begin{equation}
v_n \ \sim \ \mathcal D_2,
\end{equation}
where $\mathcal D_2$ is a sub-exponential distribution with mean $v_{\mathcal D_2}$;
\item The normalization constants $c_n$ are random variables chosen to satisfy
\begin{equation}
\label{eq:normconst}
     c_n = \prod_{\ell=1}^N (\lambda_n -\wo{\lambda_\ell})\prod_{\genfrac{}{}{0pt}{}{\ell=1}{\ell\neq n}}^N\frac{1}{\lambda_n -\lambda_\ell}.
 \end{equation}
\end{enumerate}
\end{assumption}
We will refer to this set of assumptions as ``Painlev\'e-III Rogue Wave assumptions''. The importance of the choice of normalization constants \eqref{eq:normconst} is explained in the following lemma.
\begin{lemma}
\label{lem:extremalchar}
    Suppose $\{\lambda_{n},\}_{n=1}^{N}$ is an arbitrary collection of eigenvalues in $\mathbb{C}_{+}$, and that the normalization constants are given by \eqref{eq:normconst}.  Then the associated $N$-soliton solution is extremal, and satisfies
    \begin{equation}
        \left| \psi_{N}(0,0)\right| = 2 \Im\left( \sum_{n=1}^{N} \lambda_{n} \right).
    \end{equation} 
\end{lemma}
\noindent This result is a direct consequence of Theorem \ref{thm:relation}, and its proof is presented in Section \ref{sec:DBmethod}.
Hence, for each positive integer $N$, under the Painlev\'{e}-III Rogue Wave assumptions, any realization of the scattering data corresponds to an extremal $N$-soliton solution to the NLS equation, which we will denote by 
\begin{eqnarray}
\psi_{N,III}(x,t) = \psi(x,t;\mathcal{D}_{N,III}).
\end{eqnarray}
The following theorem is our first universality result.
\begin{theorem}
\label{thm:main1}
    Under the Painlev\'e-III Rogue Wave assumptions, for all $(X,T)$ in a compact set $\fK$, there exists an $N_0>0$, a constant $C\equiv C(N_0,X,T)>0$ and $\varepsilon>0$ such that  for all $N>N_0$ the  following holds
    \begin{equation}
        \meanval{\left\vert\frac{2}{N\mu_{\cD_1}}\psi_{N,III}\left(\frac{2X}{N\mu_{\cD_1}},\frac{4T}{N^2\mu^2_{\cD_1}}\right) - \Psi_{III}(X,T)\right\vert } \leq CN^{-\varepsilon}\,,
    \end{equation}
    where the mean value is taken with respect to the product measure induced by the distribution of $\{\mu_n,v_n\}_{n=1}^N$. The quantity $\Psi_{III}(X,T)$ is a relatively new fundamental solution of the NLS equation, discovered by Bilman, Ling, and Miller \cite{Bilman2020}, described below.  Furthermore, for all $N>0$, the solution satisfies $|\psi_{N,III}(0,0)| = 2\sum_{j=1}^N\mu_j$, so that $\psi_{N,III}(x,t)$ is an extremal solution.
    
\end{theorem}

    In \cite{Bilman2020}, the authors considered a sequence of solutions of the NLS equation characterized by a meromorphic Riemann-Hilbert problem with a single pole of large order (and its complex conjugate). Under suitable rescaling and transformations as the order of the pole diverges to $\infty$, their work yielded a new limiting Riemann-Hilbert problem, and their analysis of that problem led to their discovery of $\Psi_{III}(X,T)$.  They established many fundamental properties of this solution, which we have collected and summarized in Theorem~\ref{eq:thrm38}.  In particular, they proved that $\Psi_{III}(X,T)$ solves the focusing NLS equation \eqref{eq:NLS}, and forged an explicit connection to the  Painlev\'{e}-III equation at $T=0$ via the relation
    \begin{equation}\label{eq:psi3a}
    \Psi_{III}\left(-\frac{x^2}{8},0\right) = \frac{\kappa}{x^2} \, \mathrm{exp} \left( \int_{1}^{x} \frac{2}{u(s)} \, ds \right)
    \end{equation}
    where $u(x)$ is a solution of the Painlev\'e-III equation
    \begin{equation}
        \frac{d^2u}{dx^2}= \frac{1}{u}\left(\frac{du}{dx}\right)^2 - \frac{1}{x}\frac{du}{dx} + \frac{4}{x} + 4u^3 - \frac{4}{u}
    \end{equation}
    and the constant $\kappa$ in equation \eqref{eq:psi3a} is fixed to satisfy $\left| \Psi_{III}(0,0) \right|=4$.  Note that $u$ is also characterized via a Riemann-Hilbert problem (rather than specific initial or boundary conditions).  

\begin{remark}
In \eqref{eq:psi3a}, the choice of the endpoint $1$ is arbitrary, one can consider any point on the real axis except zero. 
\end{remark}
We prove Theorem~\ref{thm:main1} in Section \ref{sec:PIII}.
\begin{remark}
    The previous result is universal in the following sense. Independent of the specific distribution of the amplitudes and the velocities, $\psi_{N,III}(x,t)$ is an extremal solution and its peak is described in terms of a solution of the Painlev\'e-III equation. 
\end{remark}

\begin{remark}
    We notice that the choice of the point $(0,0)$ is completely arbitrary. Indeed one can use the translation invariance property of the NLS equation to move the point where the function  $|\psi_{N,III}(x,t)|$ reaches its maxima from $(0,0)$ to any point $(x_0,t_0)$.
\end{remark}

\begin{remark}
This result is consistent with \cite{Bilman2020}, where the same function $u(x)$ is obtained.
\end{remark}

The second set of assumptions that we consider is the following.
 \begin{assumption}[Painlev\'e-V Rogue Wave] \label{hp:random_scattering_PV}
For each fixed $N>0$, we generate the scattering data
\begin{equation}
    \mathcal{D}_{N,V} = \left\{r \equiv 0,
\left\{  
\lambda_{n}, c_{n}
\right\}_{n=1}^{N}
    \right\}
\end{equation}     
as follows. Fix $0<\zeta<1$, let the poles $\lambda_n=-\zeta \, n +v_n+i\mu_n$, and in our probabilistic calculations we assume the following:
\begin{enumerate}

\item  The amplitudes $\mu_n$ are independent identically distributed (i.i.d.) random variables 
\begin{equation}
\mu_n \ \sim \ \mathcal D_1,
\end{equation}
where $\mathcal D_1$ is sub-exponential random variable, with positive support and mean $\mu_{\mathcal D_1}$;
\item  The velocities $v_n$ are independent identically distributed (i.i.d.) random variables 
\begin{equation}
v_n \ \sim \ \mathcal D_2,
\end{equation}
where $\mathcal D_2$ is a sub-exponential distribution with mean $v_{\mathcal D_2}$;
\item The normalization constants $c_n$ are random variables chosen to satisfy
\begin{equation}
\label{eq:normsamp}
     c_n = \prod_{\ell=1}^N (\lambda_n -\wo{\lambda_\ell})\prod_{\genfrac{}{}{0pt}{}{\ell=1}{\ell\neq n}}^N\frac{1}{\lambda_n -\lambda_\ell}.
 \end{equation}
\end{enumerate}
\end{assumption}
Hence, for each positive integer $N$, these scattering data correspond to an extremal $N$-soliton solution to the NLS equation (again achieving the theoretical maximum at $(x,t) = (0,0)$), which we will denote by 
\begin{eqnarray}
\psi_{N,V}(x,t) = \psi(x,t;\mathcal{D}_{N,V}).
\end{eqnarray}
We will refer to this set of assumptions as ``Painlev\'e-V Rogue Wave assumptions''. Under this assumption, we prove our second universality result.  
\begin{theorem}
\label{thm:main2}
    Under the Painlev\'e-V Rogue Wave assumptions, for all $(X,T)$ in a compact set $\fK$, there exists an $N_0>0$, a constant $C\equiv C(N_0,X,T)>0$ and $\varepsilon>0$ such that  for all $N>N_0$ the  following holds
    \begin{equation}
        \meanval{\left\vert\frac{1}{N}\psi_{N,V}\left(\frac{X}{N},\frac{T}{N^2}\right) - \Psi_{V}(X,T)\right\vert } \leq CN^{-\varepsilon}\,,
    \end{equation}
    where the mean value is taken with respect to the product measure induced by the distribution of $\{\mu_n,v_n\}_{n=1}^N$.  Furthermore, for all $N>0$, $|\psi_{N,V}(0,0)| = 2\sum_{j=1}^N\mu_j$, so that $\psi_{N,V}(x,t)$ is an extremal solution.
    
    The quantity $\Psi_{V}(X,T)$ is a solution of the focusing NLS equation \ref{eq:NLS} characterized by Riemann-Hilbert problem \ref{rhp:PainleveV} below. For $T=0$, $\Psi_{V}(X,T)$ admits the following integral representation
    \begin{equation}
    \label{eq:psiV_as_PV}
        \Psi_V(X,0) = \frac{\kappa}{X} \exp\left(-\frac{1}{2i\zeta}\int_{1}^{2i\zeta X} \frac{u(s)}{1-u(s)}ds\right)
    \end{equation}
    where $u(x)$ satisfies the Painlev\'e-V equation
    \begin{equation}
    \label{eq:PV}
        \frac{d^2 u}{dx^2} = \left( \frac{1}{2u} \pm \frac{1}{u-1} \right) \left( \frac{du}{dx} \right)^2 - \frac{1}{x} \frac{du}{dx} -2\frac{\mu_{\cD_1}^2}{\zeta^2} \frac{(u-1)^2}{x^2} \left( u - \frac{1}{u} \right) + \frac{u}{x}  -\frac{ u(u+1)}{2(u-1)}
    \end{equation}
    and $\kappa$ is such that $|\Psi_V(0,0)|=2\mu_{\cD_1}$. 
\end{theorem}
We prove this theorem in Section \ref{sec:PV}.
\begin{remark}
The choice of the endpoint $1$ in \eqref{eq:psiV_as_PV} is arbitrary, one can consider any point on the real axis except zero. 
\end{remark}
\begin{remark}
    Riemann-Hilbert problem \ref{rhp:PainleveV} emerges as $N \to \infty$, and so the existence of a solution is not known a-priori.  Theorems \ref{thm:existence} and \ref{thm:rhp_pv} establish existence of a solution to this problem as well as the new solution $\Psi_{V}(X,T)$.  
\end{remark}

\begin{remark}
    We remark that one can choose the location of the poles in a more exotic way, for example considering $2N$ poles distributed as

\begin{equation}
    \lambda_j = -j\zeta_1 + v_j + i\mu_j\,,\quad \lambda_{j+N} = j\zeta_2 + v_{j+N} + i\mu_{j+N} \quad j=1,\ldots,N\,,
\end{equation}
where $\mu_j,v_j$ are sub-exponential random variables and the norming constants $c_j$ are chosen as in \eqref{eq:normsamp}. With this initial data, one would obtain a different universal profile related to the Painlev\'e VI equation. By further generalizing this procedure, one can get a universal profile related to the solution of any Fuchsian system with an arbitrary number of poles. 
\end{remark}

\section{The Darboux method and extremal soliton solutions}
\label{sec:DBmethod}

In this section, we briefly describe the Darboux method for constructing pure solitonic solutions of the focusing nonlinear Schr\"{o}dinger equation and explain its relevance for the characterization of extremal multi-soliton solutions. For a detailed overview of this method, we refer the reader to \cite{gelash_strongly_2018,gelash_anomalous_2020}.

The parameters used to characterize a pure solitonic solution of the NLS equation in the Darboux framework are different than (but equivalent to) the IST scattering data \eqref{eq:spectral_RHP}.  We will use the following slightly different notation for the Darboux framework's scattering data:
\begin{equation}
\label{eq:spectral_DM}
\mathfrak{D}_N(t) = \left\{ r(z,t) \equiv 0, \{\lambda_{n}, p_{n}(t) \}_{n=1}^{N} \right\} = \left\{ r(z,t) \equiv 0, \{\lambda_{n}, p_{n}e^{ -2 i t \lambda_{n}^{2}} \}_{n=1}^{N} \right\}\,.
\end{equation}
Here, $p_n$ are the Darboux parameters associated with each soliton, playing a role analogous to the normalization constants $c_n$ in the RHP formulation. The procedure is summarized as follows. Let $\mathfrak{D}_n(0)$ denote the spectral data corresponding to the first $n$ eigenvalues $\{\lambda_1, \dots, \lambda_n\}$ and associated normalization constants $\{p_1, \dots, p_n\}$. The iterative procedure starts from the trivial solution
\begin{equation}
    \psi_0(x,0) = 0,
\end{equation}
with the corresponding solution of the Zakharov--Shabat (ZS) system \eqref{eq:ZS1} 
\begin{equation}
    \Phi_0(x,z) = \begin{pmatrix}
        e^{-i z x} & 0 \\
        0 & e^{i z x}
    \end{pmatrix}.
\end{equation}
Suppose that, after $n-1$ steps, we have constructed the $(n-1)$-soliton solution $\psi_{n-1}(x,0)$ with spectral data $\mathfrak{D}_{n-1}(0)$, which consists of the first $n-1$ eigenvalues and their associated normalization constants. The corresponding solution of the ZS system is denoted by $\Phi_{n-1}(x,z)$. The spectral data $\mathfrak{D}_n(0)$ of the $n$-soliton solution is obtained by adding an additional eigenvalue $\lambda_n$ and its normalization constant $p_n$ to $\mathfrak{D}_{n-1}(0)$. In this sense, $\mathfrak{D}_{n-1}(0)$ is taken directly from $\mathfrak{D}_n(0)$ by removing the $n$-th eigenvalue and normalization constant. The $n$-soliton solution at time $t=0$ is then constructed from $\psi_{n-1}(x,0)$ via
\begin{equation}
    \psi_n(x,0) = \psi_{n-1}(x,0) + 2i(\lambda_n - \overline{\lambda_n}) \frac{\overline{q_{n1}} q_{n2}}{|\mathbf{q}_n|^2},
\end{equation}
where the vector $\mathbf{q}_n = (q_{n1}, q_{n2})^\intercal$ is determined by $\Phi_{n-1}(x,z)$ and the spectral data $\mathfrak{D}_{n}(0)$ via the relation:
\begin{equation}
    \mathbf{q}_n(x) = \overline{\Phi_{n-1}(x,\overline{\lambda_n})} 
    \begin{pmatrix} 1 \\ p_n \end{pmatrix},
\end{equation}
where $\Phi_{n-1}(x,z)$ is the solution of the ZS system corresponding to $\psi_{n-1}(x,0)$. The corresponding ZS solution for the $n$-th soliton, $\Phi_n(x,z)$, is given in terms of $\Phi_{n-1}(x,z)$ by the following relation
\begin{equation}
    \Phi_n(x,z) = \chi_n(x,z) \cdot \Phi_{n-1}(x,z), \quad 
    (\chi_n)_{m\ell} = \delta_{m\ell} + \frac{\lambda_n - \overline{\lambda_n}}{z - \lambda_n} \frac{\overline{q_{nm}} q_{n\ell}}{|\mathbf{q}_n|^2}, \quad m,\ell = 1,2,
\end{equation}
where $\delta_{m\ell}$ is the Kronecker delta.
By iterating this procedure, starting from $\psi_0$ and $\Phi_0$, one can construct a multi-soliton solution $\psi_N(x,0)$ of the NLS equation by adding one soliton at a time. This iterative dressing procedure and its properties were analyzed in \cite{GravaJenkinsMazzucaMcLaughlin2024}, where the following result was established.

Let $\psi_N(x,t)$ denote the multi-soliton solution constructed via the Darboux method with spectral data $\mathfrak{D}_N(t)$. Then the solution satisfies
\begin{equation}
    \label{eq:DM_max}
    \max_{x,t\in\mathbb{R}} |\psi_N(x,t)| \le 2 \sum_{n=1}^N \Im(\lambda_n),
\end{equation}
and choosing $p_n = 1$, for all $n=1,\dots,N$, then
\[
\max_{x,t\in\mathbb{R}} |\psi_N(x,t)| = |\psi_N(0,0)| = 2 \sum_{n=1}^N \Im(\lambda_n).
\]  
In essence, when \(\psi_N(x,t)\) is constructed via the Darboux method, there exists a one-parameter family for $p_n$, namely $p_n=e^{i \eta}$, for which the solution attains its maximal amplitude. The special choice $\eta = 0$ recovers the equality above, see also Proposition 1.1 in \cite{GravaJenkinsMazzucaMcLaughlin2024}.

Since the Darboux normalization parameters provide the natural characterization of extremal multi-soliton solutions, it is essential to establish a precise correspondence between the Darboux parameters and the IST normalization constants, to facilitate the Riemann-Hilbert analysis.  The following theorem provides this dictionary.
\begin{theorem}
\label{thm:relation}
Consider the spectral data \(D_N(t)\) associated with RHP \ref{rhp:reflection_RHP}, and \(\mathfrak{D}_N(t)\) associated with the Darboux method, defined in Eqs. \eqref{eq:spectral_RHP}--\eqref{eq:spectral_DM}, respectively. Suppose they are related via:
     \begin{equation}
     \label{eq:relation_constants}
         c_n= \frac{1}{p_n}\prod_{\ell=1}^N (\lambda_n -\wo{\lambda_\ell})\prod_{\genfrac{}{}{0pt}{}{\ell=1}{\ell\neq n}}^N\frac{1}{\lambda_n -\lambda_\ell}\,,
     \end{equation}
     then the NLS solution arising from the two methods is the same. 
\end{theorem}
\noindent The proof of this theorem is presented in Appendix \ref{eq:appB}.
\begin{remark}
In this work, we choose $p_n = e^{i \eta}$ with $\eta \in \mathbb{R}$, so that the resulting solution achieves its maximal peak at $(x,t)=(0,0)$.  This observation completes the proof of Lemma \ref{lem:extremalchar}.
\end{remark}




\section{The Painlevé--III Rogue Wave}
\label{sec:PIII}

In this section, we prove Theorem \ref{thm:main1}. First, we recall some properties of sub-exponential random variables; then, we show that Theorem \ref{thm:main1} holds \textit{deterministically} in a subset $\fQ_N\subseteq \R^N\times \R^N_+$, and finally, by leveraging the previously mentioned properties of sub-exponential random variables, we  complete the proof.

\subsection{Probabilistic Background}
Here we recall some properties of sub-exponential random variables that we use to prove our main theorem. Experts in probability can skip this part since it only contains standard results.

Sub-exponential random variables are naturally equipped with the following norm \cite{Vershynin2018}; let $\fX$ be a sub-exponential random variable, then one can define the following norm

\begin{equation}
\label{eq:norm_subexp}
    \norm{\fX}_{\text{sub}} = \inf\{K>0 \;|\; \meanval{\exp(\fX/K)}\leq 2 \}\,.
\end{equation}
Using this definition, we can state the following lemma that we use later in the proof \cite[Corollary 2.9.2]{Vershynin2018}.

\begin{lemma}[Sub-exponential Bernstein inequality]
\label{lem:bernstein}
Let $\fX_1,\ldots,\fX_N$ be independent, mean zero, sub-exponential random variables, and $\ba=(a_1,\ldots,a_N)\in\R^N$. Then, there exists a $c>0$ such that for every $t>0$ we have

\begin{equation}\label{eq:Bernineq}
    \mathbb{P}\left( \left\vert\sum_{j=1}^N a_j\fX_j\right\vert \geq t \right) \leq 2\exp\left(-c \min\left(\frac{t}{\max_j ||\fX_j||_{\text{sub}} \max_j|a_j|}; \frac{t^2}{\max_j ||\fX_j||_{\text{sub}}^2||\ba||_2^2 }\right)\right)\,,
\end{equation}
    
\end{lemma}
The following Lemma is a straightforward application of the definition of sub-exponential random variables.

 \begin{lemma}
 \label{lem:scaling_max}
     Under Assumption \ref{hp:random_scattering_PIII} or \ref{hp:random_scattering_PV}, there exists a constant $\alpha>0$ independent of $N$, such that for all $t>0$
  
    \[
         \mathbb{P}\left( \max_{j=1,\ldots,N} \mu_j  \geq  \alpha\left( \ln(N) + t \right) + \mu_{\cD_1} \right) \leq 
             2e^{-t} \,.
     \]
        
     An analogous statement applies to the $v_j$.
 \end{lemma}

 \begin{proof}
    Let $s>0$, then, by applying standard union bounds and the definition of sub-exponential random variable, there exists a constant $\alpha>0$ independent of $N$ such that the following holds

     \begin{equation}
         \mathbb{P}\left(\max_{j=1,\ldots,N}\mu_j > s\right) \leq \sum_{j=1}^N \mathbb{P}(\mu_j - \mu_{\cD_1}>s - \mu_{\cD_1} ) \leq 2N \exp\left( - \frac{s-\mu_{\cD_1}}{\alpha}\right)\,,
     \end{equation}
     setting $s=\alpha (\ln(N) + t) + \mu_{\cD_1} $ we conclude.
 \end{proof}
 
\begin{remark}
Notice that in the previous two lemmas $t$ can depend on $N$. This implies in particular that in probability 
 \begin{equation}
     \mu_j \ll N^\varepsilon, \quad v_j \ll N^\varepsilon\,, \quad j=1,\ldots,N
 \end{equation}
 for any $\varepsilon>0$. 
\end{remark}

 \subsection{Reduction to the model problem for Painlev\'e-III}
 We now use the previous results to define the set $\fQ_N\subseteq \R^N\times \R_+^N$, where we show that a deterministic version of Theorem \ref{thm:main1} holds.
   We fix $\frac{1}{4}<\delta<\frac{1}{2}$ and define the following sets 
\begin{equation}
\label{eq:omega_delta}
    \Omega_{\delta, \mu} = \left\{ \{\mu_j\}_{j=1}^N \,\vert \, \max_{j}|\mu_j|<N^\delta\right\}\,, \quad     \Omega_{\delta, v} = \left\{ \{v_j\}_{j=1}^N \,\vert \, \max_{j}|v_j|<N^\delta\right\}, \quad \Omega_{\delta}=\Omega_{\delta,v}\times\Omega_{\delta,\mu}\,,
\end{equation}
\begin{equation}
\label{eq:V_2delta}
    V_{2 \delta}= \left\{ (\mu_1,\ldots,\mu_N)\in \R_+^N \Big\vert \Big|\sum_{j=1}^N \mu_j - N\mu_{\cD_1} \Big|\leq N^{2 \, \delta} \right\}\,.
\end{equation}
\begin{remark}
    We notice that in view of Lemma \ref{lem:scaling_max} we know that there exists a constant $C_1>0$ such that
\begin{equation}
\label{eq:small_complement}
    \mathbb{P}(\Omega^c_{\delta, \mu})<\exp(-C_1N^\delta + \ln(N))\,, \quad \mathbb{P}(\Omega^c_{\delta, v})<\exp(-C_1N^\delta + \ln(N))
\end{equation}
where obviously the term $\ln(N)$ can be absorbed by adjusting the constant $C_{1}$.  By Lemma \ref{lem:bernstein}, there exist $\varepsilon>0$ and $C_2$ such that
\begin{equation}
\label{eq:measure_v2delta}
    \mathbb{P}(V_{2\delta}^{c})<C_2e^{-N^{\varepsilon}}.
\end{equation}
\end{remark}

The aim of this subsection is to prove the following

\begin{theorem}
\label{thm:main1_deterministic}
    Fix $\frac{1}{4}<\delta<\frac{1}{2}$, and define $\fQ_N \coloneqq \Omega_{\delta} \cap (\R^N\times V_{2\delta})$, where $\Omega_\delta$ and $V_{2 \delta}$ are given by equations \eqref{eq:omega_delta}-\eqref{eq:V_2delta}, respectively. Under the Painlev\'e-III Rogue Wave assumptions, and assuming that $(\bv,\boldsymbol{\mu})\in \fQ_N$; for all $(X,T)$ in a compact set $\fK$, there exists an $N_0>0$, a constant $C\equiv C(N_0,X,T)>0$ and $\varepsilon>0$ such that  for all $N>N_0$ the  following holds
    \begin{equation}
        \left\vert\frac{2}{N\mu_{\cD_1}}\psi_{N,III}\left(\frac{2X}{N\mu_{\cD_1}},\frac{4T}{N^2\mu^2_{\cD_1}}\right) - \Psi_{III}(X,T)\right\vert  \leq CN^{-\varepsilon}\,,
    \end{equation}
    where $\Psi_{III}(X,T)$ is as in Theorem \ref{thm:main1}.
\end{theorem}

\subsubsection{Proof of the Deterministic Result}
To prove the previous theorem, we use the RHP formulation of the inverse scattering of the NLS equations. Therefore, we consider RHP \ref{rhp:reflection_RHP} which under our assumptions reads as follows:
\begin{RHP}
\label{rhp:PIII_Initial}
    Under the assumptions of Theorem \ref{thm:main1_deterministic}, find a $2 \times 2$ matrix-valued function $M(z)$ such that:
\begin{itemize}
	\item $\ {M}(z)$ is meromorphic in  $ \C$, with simple poles at the points $\{ \lambda_{j}, \overline{\lambda_{j}}\}_{j=1}^{N}$.

	\item At each of the poles $\lambda_{j}$ in the upper half-plane, ${M}(z)$ satisfies the residue condition 
	\begin{equation}
		\Res_{\lambda_j} M(z) = \lim_{z \to \lambda_j } \left[ M(z)\begin{pmatrix}
			0 & 0 \\
			c_j e^{2i\theta(\lambda_j;x,t)} & 0
		\end{pmatrix}\right]\,,
    \end{equation}
where $\theta(z;x,t) = zx + z^2t$, and at each pole $\overline{\lambda_{j}}$ in the lower half-plane:
\begin{equation}
\Res_{\wo {\lambda_j}} M(z) = \lim_{z \to \wo \lambda_j } \left[ M(z) \begin{pmatrix}
			0 & -\wo c_j e^{-2i\theta(\wo \lambda_j;x,t)} \\
			0 & 0
		\end{pmatrix}\right].
	\end{equation}
\item As $z\to\infty$, $M(z)$ has the following expansion 
\begin{equation}
    M(z;x,t) = I + \frac{1}{2 i z} \pmtwo{-m_{N,III}(x,t)}{\psi_{N,III}(x,t)}{\overline{\psi_{N,III}(x,t)}}{m_{N,III}(x,t)}
+ \mathcal{O} \left(\frac{1}{z^{2}} \right)
\end{equation}
where $m_{N,III}(x,t) = \int_{x}^{\infty} |\psi_{N,III}(s,t)|^{2}\di s$.
\end{itemize}
\end{RHP}

 Since $(\boldsymbol{\mu},\bv)\in \fQ_N \coloneqq \Omega_{\delta} \cap (\R^N\times V_{2\delta})$, for $N$ large enough, we can consider a circle $\Gamma$ of radius $N\mu_{\cD_1}/2$ centered at the origin enclosing all the poles  $\lambda_j$.  We define the following transformation
 \begin{equation}
 \label{eq:from_M_to_B}
     B(z) = \begin{cases}
M(z) \, \frac{1}{\sqrt{2}}\begin{pmatrix}
1 & a(z) e^{-2i \theta(z)}\\
-a^{-1}(z) e^{2i \theta(z)} & 1
\end{pmatrix}, & \quad z \,\, \text{inside} \,\, \Gamma \\\\[2pt]
M(z), & \quad z \,\, \text{outside} \,\, \Gamma
\end{cases}
 \end{equation}
where we introduced the following notation for the Blaschke factor $a(z)$

\begin{equation}
\label{eq:blash}
    a(z)=\prod_{j=1}^N \frac{z-\lambda_j}{z - \wo \lambda_j}\,.
\end{equation}
By direct calculation, one can prove that the matrix $B(z)$ satisfies the following RHP.

 \begin{RHP}
\label{rhp:RHP_B}
Find a $2 \times 2$ matrix-valued function $B(z)$ such that:
\begin{itemize}
\item $B(z)$ is analytic in  $ \mathbb{C} \setminus \Gamma $ with jump given by:
\[B_{+}(z)=B_{-}(z)\frac{1}{\sqrt{2}} \begin{pmatrix}
1 & a(z) e^{-2i \theta(z)}\\
-a^{-1}(z) e^{2i \theta(z)} & 1
\end{pmatrix}, \quad z \in \Gamma\]
\item $B(z)$ has the following expansion as $z\to\infty$
\begin{equation}
    B(z) = \  I + \frac{1}{2iz}\begin{pmatrix}
    -m_{N,III}(x,t) & \psi_{N,III}(x,t) \\
    \wo\psi_{N,III}(x,t) & m_{N,III}(x,t)
\end{pmatrix} + \cO(z^{-2}).
\end{equation}
\end{itemize}
\end{RHP}
Consider now the following scaling
\begin{equation}
\label{eq:scaling}
    z= \frac{N\mu_{\cD_1}Z}{2}\,, \quad x = \frac{2X}{N\mu_{\cD_1}}\,, \quad t = \frac{4T}{\mu_{\cD_1}^2N^2}\,.
\end{equation}
In this regime, one immediately gets that for $z\in \Gamma$:
\begin{align*}
\ln(a(z)) &= \ln(a\left(NZ\mu_{\cD_1}/2\right)) \\
&= \sum_{j=1}^N \ln \left( \frac{N Z \mu_{\cD_1}}{2} - v_j - i \mu_j \right) - \ln \left( \frac{N Z \mu_{\cD_1}}{2} - v_j + i \mu_j \right) \\
&= - \frac{4i}{N Z \mu_{\cD_1}} \sum_{j=1}^N \mu_j + \cO \left( N^{2 \delta -1}\right)\\
&= - \frac{4i}{Z} + \cO \left( N^{2 \delta -1}\right).
\end{align*}

In the (new) complex $Z$ plane, the contour $\Gamma$ is the unit circle.  Therefore, in the scaling regime \eqref{eq:scaling}, for $|Z|=1$, the jump matrix for $B$ has the following asymptotic behavior:
\begin{equation}
\begin{split}
 &   \frac{1}{\sqrt{2}} \pmtwo{1}{a(z) e^{-2i \theta(z)}}{-a^{-1}(z) e^{2i \theta(z)}}{1}=
 \\& =  e^{-i\left(XZ + TZ^2 + 2 Z^{-1} \right)\sigma_3} \begin{pmatrix}
    \frac{1}{\sqrt{2}} & \frac{1}{\sqrt{2}}(1+\cO(N^{-1+2\delta})\\
    -\frac{1}{\sqrt{2}}(1+\cO(N^{-1+2\delta}) & \frac{1}{\sqrt{2}}
\end{pmatrix}e^{i\left(XZ + TZ^2 + 2 Z^{-1} \right)\sigma_3},
\end{split}
\end{equation}
where the error terms $\cO(N^{-1+2\delta})$ are uniform for all $(\bv,\boldsymbol{\mu})\in \fQ_N$, all $X,T \in \mathfrak{K}$ and all $|Z|=1$. 
\begin{remark}
    In fact the error terms $\cO(N^{-1+2\delta})$ are uniform for all $\frac{7}{8} < |Z| <\frac{9}{8} $.  This permits slight deformations of the unit circle for technical aspects of the Riemann-Hilbert analysis.
\end{remark}

This leads us to introduce the model RHP for a function $R(Z,X,T)$, by dropping the terms $\cO(N^{-1 + 2 \delta})$ in the above jump matrix.  This RHP was introduced in \cite[RHP 4]{Bilman2020}.  We will show that $R$ is a good approximation to $B(z(Z), x(X), t(T))$.

\begin{RHP}
\label{rhp:PainleveIII}
    Let $(X,T)\in\R^2$ be two arbitrary parameters. Find a $2\times2$ matrix $R(Z;X,T)\equiv R(Z)$ such that:

    \begin{enumerate}
        \item $R(Z;X,T)$ is analytic for $|Z|\neq1$, and takes continuous boundary values in the interior and in the exterior of the circle.
        \item The two boundary values are related by 
        \begin{equation}
            R_+(Z) = R_-(Z)e^{-i(XZ+TZ^2+2Z^{-1})\sigma_3}\begin{pmatrix}
                \frac{1}{\sqrt{2}} & \frac{1}{\sqrt{2}}\\
                -\frac{1}{\sqrt{2}} & \frac{1}{\sqrt{2}}
            \end{pmatrix}e^{i(XZ+TZ^2+2Z^{-1})\sigma_3}, \qquad |Z|=1.
        \end{equation}
        \item $R(Z;X,T)$ has the following expansion as $Z\to\infty$
        
        \begin{equation}
            R(Z;X,T) = I +  \cO(Z^{-1}).
        \end{equation}
        
    \end{enumerate}
\end{RHP}
In \cite{Bilman2020}, the authors proved the following result for the above RHP. In \cite{Bilman2024}, they further established existence for all $X,T$, and asymptotic behavior for $X$ large and $T=0$.

\begin{theorem}\label{eq:thrm38}
RHP \ref{rhp:PainleveIII} has a unique solution for $(X,T)\subseteq \mathfrak{K}$, where $\mathfrak{K}\subseteq \R^2$ is a compact set. In particular, the following expansion holds
     \begin{equation}
         R(Z) =  I+\frac{1}{2iZ}\begin{pmatrix}
                -\fm_{III}(X,T) & \Psi_{III}(X,T) \\
                \wo{\Psi}_{III}(X,T) & \fm_{III}(X,T)
            \end{pmatrix} + \cO(Z^{-2}), \quad Z \to\infty
        \end{equation}
        where $\fm_{III}(X,T) = \int_X^{+\infty} |\Psi_{III}(s,T)|^2\di s$. $\Psi_{III}(X,T)$ solves the focusing NLS equation \eqref{eq:NLS}, and it is given explicitly in terms of the function $u(x)$ via the following relation
    \begin{equation}\label{eq:psi3}
    \Psi_{III}\left(-\frac{x^2}{8},0\right) = \frac{\kappa}{x^2} \, \mathrm{exp} \left( \int_{1}^{x} \frac{2}{u(s)} \, ds \right)
    \end{equation}
    where $u(x)$ satisfies the Painlev\'e-III equation
    \begin{equation}
        \frac{d^2u}{dx^2}= \frac{1}{u}\left(\frac{du}{dx}\right)^2 - \frac{1}{x}\frac{du}{dx} + \frac{4}{x} + 4u^3 - \frac{4}{u}
    \end{equation}
    and the constant $\kappa$ in equation \eqref{eq:psi3} is fixed to satisfy $\left| \Psi_{III}(0,0) \right|=4$.    
\end{theorem}

\begin{remark}
    We notice that we collected in one statement multiple results presented in \cite{Bilman2020}.
\end{remark}
To conclude the proof of Theorem \ref{thm:main1_deterministic}, we show a \textit{small norm argument} for $E(Z)=B(Z)R(Z)^{-1}$. Specifically, we consider  the matrix $E(Z;X,T)\equiv E(Z;X,T)=B(Z)R(Z;X,T)^{-1}$, which, by direct calculation, solves the following RHP.

\begin{RHP}
\label{rhp:E}
        Let $(X,T)\in\R^2$ be two arbitrary parameters. Find a $2\times2$ matrix $E(Z;X,T)\equiv E(Z)$ such that:
        \begin{enumerate}
        \item $E(Z)$ is analytic for $|Z|\neq1$, and takes continuous boundary values on $|Z|=1$ from the interior and the exterior of the circle.
        \item The two boundary values are related by 
        \begin{align}
            & E_+(Z) = E_-(Z) (I +W_N(Z;X,T)), \qquad |Z|=1\\
            & W_N(Z;X,T) =\wt R_-(Z)\begin{pmatrix}
                 \cO(N^{2\delta-1}) &\cO(N^{2\delta-1})\\
                \cO(N^{2\delta-1}) & \cO(N^{2\delta-1})
            \end{pmatrix} \wt R_-(Z)^{-1},\\
            & \wt R(Z) = R(Z) \exp\left(2i(XZ + TZ^2 -2Z^{-1} )\sigma_3 \right)
            \end{align}
        \item $E(Z;X,T)$ has the following expansion as $Z\to\infty$
        
        \begin{equation}
            E(Z;X,T) = I + \frac{1}{2iZ}\begin{pmatrix}
                \fm_{III}(X,T)-\fm_{N,III}(X,T) & \Psi_{III}(X,T)- \Psi_{N,III}(X,T) \\
                \wo\Psi_{III}(X,T)-\wo{\Psi}_{N,III}(X,T) &  \fm_{N,III}(X,T) -\fm_{III}(X,T)
            \end{pmatrix} + \cO(Z^{-2}).
        \end{equation}
    \end{enumerate}
\end{RHP}

where $\Psi_{N,III}(X,T) \coloneqq \frac{2}{N\mu_{\cD_1}} \psi_{N,III}\left(\frac{2X}{N\mu_{\cD_1}},\frac{4T}{\mu_{\cD_1}^2N^2}\right)$ and analogously for $\fm_{N,III}(X,T)$.

\begin{remark}
We claim that $W_N(Z)$ in RHP~\ref{rhp:E} is $\cO(N^{2\delta-1})$. To justify this, we show that the function $R(Z)$ solving RHP~\ref{rhp:PainleveIII}, and consequently the function $\wt R(Z)$ defined in RHP~\ref{rhp:E}, is \emph{uniformly bounded} on an open annular neighborhood of the unit circle  for $(X,T) \in \mathfrak{K}$. Even though $R$ has a jump across the unit circle, it achieves its boundary values in the sense of continuous functions, and so is bounded even up to the circle. In \cite{Bilman2020}, the authors introduce the function $R^{\pm}(Z)$ defined by
\begin{equation}\label{eq:dfnR}
R^{\pm}(Z;X,T) =
\begin{cases}
P^{\pm}(Z;X,T) \, e^{-i(Z X + Z^2 T)\,\sigma_3} 
\begin{pmatrix} \frac{1}{\sqrt{2}} & -\frac{1}{\sqrt{2}} \\[1mm] \frac{1}{\sqrt{2}} & \frac{1}{\sqrt{2}} \end{pmatrix} 
e^{i(Z X + Z^2 T)\,\sigma_3}, & |Z| < 1,\\[1mm]
P^{\pm}(Z;X,T) \, e^{\pm 2i/Z \, \sigma_3}, & |Z| > 1,
\end{cases}
\end{equation}
which coincides with the function $R(Z)$ that we consider here up to different orientation of the contour (see the comparison between RHP~4 in \cite{Bilman2020} and RHP~\ref{rhp:PainleveIII}). Moreover, the authors show that $P^{\pm}(Z)$ satisfies
\begin{equation}\label{eq:bound}
\sup_{\substack{|Z| \neq 1, \, (X,T) \in \mathfrak{K}}} \| P^{\pm}(Z;X,T) \| = C_{\mathfrak{K}} < \infty.
\end{equation}
It then follows from definition \eqref{eq:dfnR} and equation \eqref{eq:bound} that $\wt R(Z)$ is bounded on an open annular neighborhood of the unit circle, uniformly for $(X,T) \in \mathfrak{K}$, which establishes the claim.
\end{remark}
Next, the following Lemma, whose proof can be found in the appendix, holds.
\begin{lemma}
\label{lem:small_norm_PIII}
    Consider RHP \ref{rhp:E}. It has a unique solution and 
    \begin{equation}
        E(Z) = I + \cO\left(\frac{1}{Z N^{1-2\delta}}\right), \quad  Z \to\infty\,.
    \end{equation}
\end{lemma}
Applying the previous Lemma, we deduce that there exists a constant $C$ such that
\begin{equation}
    |\Psi_{III}(X,T)- \Psi_{N,III}(X,T)| \leq C N^{2\delta -1}\,,
\end{equation}
which completes the proof of Theorem \ref{thm:main1_deterministic}. \qed
\subsection{Proof of Theorem \ref{thm:main1}}
We are now ready to prove Theorem \ref{thm:main1}.
From Theorem \ref{thm:main1_deterministic}, we know that there exist $N_0\in\N,\delta\in \left(\frac{1}{4},\frac 12\right), C>0$, such that for all $N>N_0$
    \begin{align}
       & \mathbb{E}\bigg[\bigg\vert\frac{2}{N\mu_{\cD_1}} \psi_{N,III}\left(\frac{2X}{N\mu_{\cD_1}},\frac{4T}{\mu_{\cD_1}^2N^2}\right)- \Psi_{III}(X,T)\bigg\vert\bigg]\\ & \qquad = \int_{\Omega_\delta\cap (V_{2\delta}\times \R^N)} \left\vert\frac{2}{N\mu_{\cD_1}} \psi_{N,III}\left(\frac{2X}{N\mu_{\cD_1}},\frac{4T}{\mu_{\cD_1}^2N^2}\right) - \Psi_{III}(X,T)\right\vert\di \boldsymbol{v} \, \di \boldsymbol{\mu}\\ & \qquad \quad + \int_{(\Omega_\delta\cap (V_{2\delta}\times \R^N))^c} \left\vert\frac{2}{N\mu_{\cD_1}} \psi_{N,III}\left(\frac{2X}{N\mu_{\cD_1}},\frac{4T}{\mu_{\cD_1}^2N^2}\right) - \Psi_{III}(X,T)\right \vert\di \boldsymbol{v} \, \di \boldsymbol{\mu}\\
        &  \qquad \leq C N^{2\delta-1} +  \int_{\Omega_\delta^c} \left\vert\frac{2}{N\mu_{\cD_1}} \psi_{N,III}\left(\frac{2X}{N\mu_{\cD_1}},\frac{4T}{\mu_{\cD_1}^2N^2}\right) - \Psi_{III}(X,T)\right \vert\di \boldsymbol{v} \, \di \boldsymbol{\mu} \\
        &\qquad \quad +\int_{ V_{2 \delta}^c \times \mathbb{R}^{N}} \left\vert\frac{2}{N\mu_{\cD_1}} \psi_{N,III}\left(\frac{2X}{N\mu_{\cD_1}},\frac{4T}{\mu_{\cD_1}^2N^2}\right) - \Psi_{III}(X,T)\right \vert\di \boldsymbol{v} \, \di \boldsymbol{\mu}\\
    \end{align}
\begin{align}
 \qquad \quad & \leq C N^{2\delta-1} + \int_{\Omega_\delta^c} \left\vert\frac{2}{N\mu_{\cD_1}} \psi_{N,III}\left(\frac{2X}{N\mu_{\cD_1}},\frac{4T}{\mu_{\cD_1}^2N^2}\right)\right\vert\di \boldsymbol{v} \, \di \boldsymbol{\mu} \\  \\
        & \quad + \int_{\Omega_\delta^c}\left\vert \Psi_{III}(X,T)\right \vert\di \boldsymbol{v} \, \di \boldsymbol{\mu} +\int_{  V_{2 \delta}^c \times \mathbb{R}^{N}} \left\vert\frac{2}{N\mu_{\cD_1}} \psi_{N,III}\left(\frac{2X}{N\mu_{\cD_1}},\frac{4T}{\mu_{\cD_1}^2N^2}\right)\right\vert\di \boldsymbol{\mu}\, \di \boldsymbol{v}     \label{ineq:LongCalc} \\ & \quad + \int_{ V_{2 \delta}^c \times \mathbb{R}^{N}}\left\vert \Psi_{III}(X,T)\right \vert\di \boldsymbol{v} \, \di \boldsymbol{\mu}.
\end{align}
    
Since $\Psi_{III}(X,T)$ is a continuous function and we are considering $(X,T)$ in a compact set, by Heine-Borel theorem there exists a constant $C_1$ such that $|\Psi_{III}(X,T)|\leq C_1$, and there exists a constant $C_2$, $\varepsilon >0$ 
    \begin{equation}
\mathbb{P}(\Omega_\delta^c)+\mathbb{P}( V_{2 \delta}^c \times \mathbb{R}^{N})\leq \exp(-C_2 N^{\varepsilon})\,,
    \end{equation}
    therefore we deduce that there exist two constants $c,\epsilon>0$ such that
    \begin{equation}
       \int_{ V_{2 \delta}^c \times \mathbb{R}^{N}}\left\vert \Psi_{III}(X,T)\right \vert\di \boldsymbol{v} \, \di \boldsymbol{\mu} + \int_{\Omega_\delta^c}\left\vert \Psi_{III}(X,T)\right \vert\di \boldsymbol{v} \, \di \boldsymbol{\mu} \leq cN^{-\epsilon}\,. 
    \end{equation}
    From \eqref{eq:DM_max}, we deduce that
    \begin{equation}
    \begin{split}
                \int_{ \Omega_\delta^c} \left\vert\frac{2}{N\mu_{\cD_1}} \psi_{N,III}\left(\frac{2X}{N\mu_{\cD_1}},\frac{4T}{\mu_{\cD_1}^2N^2}\right)\right\vert\di \boldsymbol{v} \, \di \boldsymbol{\mu} \leq \frac{4}{N\mu_{\cD_1}} \int_{ \Omega_{\delta,\mu}^c \times \R^{N}} \sum_{j=1}^N\mu_j\di \boldsymbol{v} \, \di \boldsymbol{\mu} \leq \frac{4}{\mu_{\cD_1}}\int_{\Omega_{\delta,\mu}^c} \max_j \mu_j \di \boldsymbol{\mu}.
    \end{split}
    \end{equation}
We introduce the notation $Y\sim\max_j \mu_j$ and $\sigma(Y)= \mathbb{P}(Y>y)$, then we have that:
\begin{equation}
        \int_{ \Omega_\delta^c} \left\vert\frac{2}{N\mu_{\cD_1}} \psi_{N,III}\left(\frac{2X}{N\mu_{\cD_1}},\frac{4T}{\mu_{\cD_1}^2N^2}\right)\right\vert\di \boldsymbol{v} \, \di \boldsymbol{\mu}  \leq \frac{4}{\mu_{\cD_1}}\int_{N^{\delta}}^{\infty} Y \, \di \sigma(Y).
    \end{equation}
    Integrating by parts and using the definition of the improper integral, we get
    \begin{equation}
    \int_{N^{\delta}}^{\infty} Y \, \di \sigma(Y) = -\lim_{t \to \infty} t \, \mathbb{P}(Y>t) + N^\delta\mathbb{P}(Y>N^\delta)+ \int_{N^\delta}^\infty \sigma(Y) \di Y\,.
    \end{equation}
    By Lemma \ref{lem:scaling_max}, we conclude that there exist two constants $c,\epsilon>0$, such that 

    \begin{equation}
        \int_{ \Omega_\delta^c} \left\vert\frac{2}{N\mu_{\cD_1}} \psi_{N,III}\left(\frac{2X}{N\mu_{\cD_1}},\frac{4T}{\mu_{\cD_1}^2N^2}\right)\right\vert\di \boldsymbol{v} \, \di \boldsymbol{\mu}\leq cN^{-\epsilon}\,.
    \end{equation}
    Regarding the last term in \eqref{ineq:LongCalc}, we notice that it can be estimated as 

    \begin{align*}
    &\int_{  V_{2 \delta}^c \times \mathbb{R}^{N}} \left\vert\frac{2}{N\mu_{\cD_1}} \psi_{N,III}\left(\frac{2X}{N\mu_{\cD_1}},\frac{4T}{\mu_{\cD_1}^2N^2}\right)\right\vert\di \boldsymbol{\mu}\, \di \boldsymbol{v}\\
      & = \int_{ V_{2\delta}^c} \left\vert\frac{2}{N\mu_{\cD_1}} \psi_{N,III}\left(\frac{2X}{N\mu_{\cD_1}},\frac{4T}{\mu_{\cD_1}^2N^2}\right)\right\vert\di \boldsymbol{\mu} \leq \frac{4}{N\mu_{\cD_1}}\left(\int_{V_{2 \delta}^c}\sum_{j=1}^N \mu_j - N\mu_{\cD_1}\di\boldsymbol{\mu} + N\mu_{\cD_1}\int_{V_{2 \delta}^c}\di \boldsymbol{\mu}\right)\,.
    \end{align*}
    Let $Y= \sum_{j=1}^N \mu_j - N\mu_{\cD_1}$ (note we are reusing $Y$, hopefully without confusion) and let $\sigma(Y)$ be its cumulative distribution (so $\di \sigma(Y)$ would be its p.d.f.). As before, integrating by parts we get
\begin{equation}
\begin{split}
    \int_{V_{2 \delta}^c}\sum_{j=1}^N \mu_j - N\mu_{\cD_1}\di\boldsymbol{\mu} = &\int_{N^{2 \delta}}^\infty Y \di \sigma(Y)\\ =&-\lim_{t \to\infty}t \, \mathbb{P}(Y>t) + N^{2 \delta} \mathbb{P}\left(Y>N^{2 \delta}\right)+ \int_{N^{2 \delta}}^\infty \sigma(Y) \di Y.
\end{split}
\end{equation}
By Lemma \ref{lem:bernstein}, we deduce that the right hand side of the previous equation goes to $0$ as $N$ grows and by \eqref{eq:measure_v2delta} $\mathbb{P}(V_{2 \delta}^c)\xrightarrow{N\to\infty} 0$.  This completes the proof of Theorem \ref{thm:main1}.
\qed

\section{The Painlevé--V Rogue Wave}
\label{sec:PV}

In this section, we prove Theorem \ref{thm:main2}. The structure of the proof follows the same line as the previous one. We first identify $\fP_N\subseteq \R^N\times \R^N_+$ where the result holds \textit{deterministically}, and then, using some probabilistic estimates, we complete the proof. The main difference with the previous case is that we introduce a different model RHP, that we analyze proving existence and uniqueness of the solution and connecting it with the solution of the Painlevé--V equation \eqref{eq:PV}, see Theorems \ref{thm:existence} and \ref{thm:rhp_pv} below.
Furthermore, we need a more sophisticated argument to show that the complementary set $\fP_N^c$ has small measure, specifically, we use the following Lemma, whose proof can be found in Appendix \ref{app:epsilon_net}.

\begin{lemma}
\label{lem:last_boredoom}
   Fix $\zeta\in (0,1)$, and $\frac{1}{4}< \delta$. Let $X_1,\ldots, X_N$ be positive i.i.d. sub-exponential random variable with mean $\mu$. Then, there exists an $N_0>0$ and two positive constants $c,\gamma>0$ independent of $N$ such that  

   \begin{equation}
       \mathbb{P}(\cU^c)\leq e^{-cN^\gamma}\,,
   \end{equation}
   where
   \begin{equation}
       \cU =\left\{ (X_1,\ldots,X_N) \in \R_+^N \Big\vert\,\, \forall Z: |Z|=1, \Bigg|\sum_{j=1}^N \frac{  X_j - \mu }{Z+\zeta \frac{j}{N}}\Bigg|\leq N^{2 \, \delta} \right\}\,.
   \end{equation}
\end{lemma}

\subsection{Reduction to the model problem for Painlevé--V}
 We define a set $\fP_N \subseteq\R^N\times\R^N_+$ where we show that a deterministic version of Theorem \ref{thm:main2} holds. Fix $\frac{1}{4}<\delta <\frac{1}{2}$ and define the following set

\begin{equation}\label{eq:setU}
    \mathcal{U}_{2 \delta}= \left\{ (\mu_1,\ldots,\mu_N) \in \R_+^N \Big\vert\,\, \forall Z: |Z|=1, \Bigg|\sum_{j=1}^N \frac{  \mu_j - \mu_{\cD_1} }{Z+\zeta \frac{j}{N}}\Bigg|\leq N^{2 \, \delta} \right\}\,.
\end{equation} 
\begin{remark}
    We notice that applying Lemma \ref{lem:last_boredoom}, we deduce that there exist $c>0,\epsilon>0$ such that
    \begin{equation}
        \mathbb{P}(\cU_{2\delta}^c)\leq e^{-cN^\epsilon}
    \end{equation}
    where $\mathbb{P}$ is the product measure of the random variables $(\bv,\boldsymbol{\mu})$.
\end{remark}
The aim of this subsection is to prove the following theorem.

\begin{theorem}
\label{thm:main2_deterministic}
    Fix $\frac{1}{4}<\delta<\frac{1}{2}$, and define $\fP_N \coloneqq \Omega_{\delta} \cap (\R^N\times \cU_{2\delta})$ \eqref{eq:omega_delta}-\eqref{eq:setU}. Under the Painlev\'e-V Rogue Wave assumptions, and assuming that $(\bv,\boldsymbol{\mu})\in \fP_N$; for all $(X,T)$ in a compact set $\fK$, there exists an $N_0>0$, a constant $C\equiv C(N_0,X,T)>0$ and $\varepsilon>0$ such that  for all $N>N_0$ the  following holds
    \begin{equation}
        \left\vert\frac{1}{N}\psi_{N,V}\left(\frac{X}{N},\frac{T}{N^2}\right) - \Psi_{V}(X,T)\right\vert  \leq CN^{-\varepsilon}\,,
    \end{equation}
    where $\Psi_{V}(X,T)$ is as in Theorem \ref{thm:main2}.
\end{theorem}

\subsubsection{Proof of the Deterministic Result}
To prove the previous theorem, we use the RHP formulation of the inverse scattering of the NLS equation. Therefore, we consider RHP \ref{rhp:reflection_RHP} under the Painlevé--V Rogue Wave assumptions.

\begin{RHP}
\label{rhp:PV_Initial}
    Under the assumption of Theorem \ref{thm:main2_deterministic}, find a $2 \times 2$ matrix-valued function $M(z)$ such that:
\begin{itemize}
	\item $\ {M}(z)$ is meromorphic in  $ \C$, with simple poles at the points $\{ \lambda_{j}, \overline{\lambda_{j}}\}_{j=1}^{N}$.

	\item At each of the poles $\lambda_{j}$ in the upper half-plane, ${M}(z)$ satisfies the residue condition 
	\begin{equation}
		\Res_{\lambda_j} M(z) = \lim_{z \to \lambda_j } \left[ M(z)\begin{pmatrix}
			0 & 0 \\
			c_j e^{2i\theta(\lambda_j;x,t)} & 0
		\end{pmatrix}\right]\,,
    \end{equation}
where $\theta(z;x,t) = zx + z^2t$, and at each pole $\overline{\lambda_{j}}$ in the lower half-plane:
\begin{equation}
\Res_{\wo {\lambda_j}} M(z) = \lim_{z \to \wo \lambda_j } \left[ M(z) \begin{pmatrix}
			0 & -\wo c_j e^{-2i\theta(\wo \lambda_j;x,t)} \\
			0 & 0
		\end{pmatrix}\right].
	\end{equation}
\item As $z\to\infty$, $M(z)$ has the following expansion 
\begin{equation}
    M(z;x,t) = I + \frac{1}{2 i z} \pmtwo{-m_{N,V}(x,t)}{\psi_{N,V}(x,t)}{\overline{\psi_{N,V}(x,t)}}{m_{N,V}(x,t)}
+ \mathcal{O} \left(\frac{1}{z^{2}} \right)
\end{equation}
where $m_{N,V}(x,t) = \int_{x}^{\infty} |\psi_{N,V}(s,t)|^{2}\di s$.
\end{itemize}
\end{RHP}
Since $(\bv,\boldsymbol{\mu})\in \fP_N$ and within the framework of the Painlev\'e-V Rogue Wave assumptions, for $N$ large enough we can enclose all the poles $\lambda_j$ by the circle $\Gamma$ of radius $N$. We define the following transformation
 \begin{equation}
 \label{eq:from_M_to_B_PV}
     B(z) = \begin{cases}
M(z) \, \frac{1}{\sqrt{2}}\begin{pmatrix}
1 & a(z) e^{-2i \theta(z)}\\
-a^{-1}(z) e^{2i \theta(z)} & 1
\end{pmatrix}, & \quad z \,\, \text{inside} \,\, \Gamma \\\\[2pt]
M(z), & \quad z \,\, \text{outside} \,\, \Gamma
\end{cases}
 \end{equation}
where $a(z)$ is defined in \eqref{eq:blash}. By direct calculation, one can prove that the matrix $B(z)$ satisfies the following RHP.
 \begin{RHP}
 \label{rhp:B_PV}
Find a $2 \times 2$ matrix-valued function $B(z)$ such that:
\begin{itemize}
\item $B(z)$ is analytic in  $ \mathbb{C} \setminus \Gamma $ with jump given by:
\[B_{+}(z)=B_{-}(z)\frac{1}{\sqrt{2}} \begin{pmatrix}
1 & a(z) e^{-2i \theta(z)}\\
-a^{-1}(z) e^{2i \theta(z)} & 1
\end{pmatrix}, \quad |z|=N.\]
\item $B(z)$ has the following expansion as $z\to\infty$
\begin{equation}
    B(z) = \  I + \frac{1}{2iz}\begin{pmatrix}
    -m_{N,V}(x,t) & \psi_{N,V}(x,t) \\
    \wo\psi_{N,V}(x,t) & m_{N,V}(x,t)
\end{pmatrix} + \cO(z^{-2}).
\end{equation}
\end{itemize}
\end{RHP}

Analogously as we did in the previous case, we consider the following scaling

\begin{equation}
    \label{eq:PV_scaling}
    z= NZ, \quad x= \frac{X}{N}, \quad t=\frac{T}{N^2}\,.
\end{equation}

Under this scaling, and since $(\bv,\boldsymbol{\mu})\in \fP_N$, we notice that the quantity $\sum_{j=1}^{N} \frac{\mu_j}{N Z + \zeta j}$ deterministically converges to $ \frac{\mu_{\cD_1}}{\zeta} \ln \left( \frac{Z+\zeta}{Z}\right)$. Specifically,
\begin{equation}
\sum_{j=1}^{N} \frac{\mu_j}{N Z + \zeta j} - \mu_{\cD_1} \frac{1}{\zeta} \ln \left( \frac{Z+\zeta}{Z}\right) = \mathcal{O}(N^{2 \delta -1}), \quad 2 \delta < 1.
\end{equation}
Therefore, following the same steps as in the previous section, we have that:
\begin{equation}
\begin{split}
        \ln(a(z)) & = \ln(a(NZ))= \sum_{j=1}^N \ln\left(NZ + \zeta j - v_j -i\mu_j\right) - \ln\left(NZ + \zeta j - v_j +i\mu_j\right)\\& = -2i\sum_{j=1}^N \frac{\mu_j}{NZ + \zeta j} + \cO(N^{2 \delta-1})\\
        &= -2i\frac{\mu_{\cD_1}}{\zeta} \ln\left(\frac{Z+\zeta}{Z}\right) + \cO(N^{2 \delta-1})\,.
\end{split}
\end{equation}
In this asymptotic regime, the jump matrix for $B(Z)$, for $Z$ such that $|Z|=1$, becomes
\begin{align*}
&\frac{1}{\sqrt{2}} \begin{pmatrix}
1 & a(z) e^{-2i \theta(z)}\\
-a^{-1}(z) e^{2i \theta(z)} & 1
\end{pmatrix} =\\ &
e^{-i\left(XZ+TZ^2+2\frac{\mu_{\cD_1}}{\zeta} \ln\left(\frac{Z+\zeta}{Z}\right)\right)\sigma_3}\begin{pmatrix}
                \frac{1}{\sqrt{2}} & \frac{1}{\sqrt{2}}\\
                -\frac{1}{\sqrt{2}} & \frac{1}{\sqrt{2}}
            \end{pmatrix}e^{i\left(XZ+TZ^2+2\frac{\mu_{\cD_1}}{\zeta} \ln\left(\frac{Z+\zeta}{Z}\right)\right)\sigma_3} + \cO(N^{2\delta-1}),
\end{align*}
where the error terms are uniform for all $|Z|=1 $ and all $(X,T) \in \mathfrak{K}$.  (Again, the error terms are actually uniform for all $Z$ in a small annular neighborhood of the unit circle.) Thus, we are naturally led to consider the following model problem.
\begin{RHP}
\label{rhp:PainleveV}
    Let $(X,T)\in\R^2$ be two arbitrary parameters. Find a $2\times2$ matrix $R(Z;X,T)$ such that:
    \begin{enumerate}
        \item $R(Z;X,T)$ is analytic for $|Z|\neq1$, and takes continuous boundary values in the interior and in the exterior of the circle.
        \item The two boundary values are related by 
        \begin{equation}
            R_+(Z) = R_-(Z)e^{-i\left(XZ+TZ^2+\frac{\mu_{\cD_1}}{\zeta} \ln\left(\frac{Z+\zeta}{Z}\right)\right)\sigma_3}\begin{pmatrix}
                \frac{1}{\sqrt{2}} & \frac{1}{\sqrt{2}}\\
                -\frac{1}{\sqrt{2}} & \frac{1}{\sqrt{2}}
            \end{pmatrix}e^{i\left(XZ+TZ^2+\frac{\mu_{\cD_1}}{\zeta} \ln\left(\frac{Z+\zeta}{Z}\right)\right)\sigma_3}, \quad |Z|=1.
        \end{equation}
        \item $R(Z;X,T)= I + \cO(Z^{-1}), \quad Z\to\infty$.
    \end{enumerate}
\end{RHP}
\begin{remark}
    We notice that in the limit as  $\zeta\to0$ the  previous RHP converges, up to a rescaling, to RHP \ref{rhp:PainleveIII}.
\end{remark}
As in the previous section, our strategy is to use the solution of the preceding RHP, whose existence and uniqueness we establish, to approximate the solution of RHP \ref{rhp:B_PV}. We now state two results. The first, proved in Appendix \ref{app:existence_theory}, guarantees that the solution to RHP \ref{rhp:PainleveV} exists and is unique. The second, proved in the next subsection, explains how RHP \ref{rhp:PainleveV} is connected to the NLS equation \eqref{eq:NLS} and to the Painlevé--V equation \eqref{eq:PV}.
\begin{theorem}
\label{thm:existence}
    Fix $(X,T)$ in a compact set, $\zeta>0,\, \mu_{\cD_1}>0$. Then RHP \ref{rhp:PainleveV} has a unique solution which is uniformly bounded for all $(X,T) \in \mathfrak{K}$.
\end{theorem}

\begin{theorem}
\label{thm:rhp_pv}
    Fix $(X,T)$ in a compact set, $\zeta>0,\, \mu_{\cD_1}>0$. Then the unique solution $R(Z)$ of RHP \ref{rhp:PainleveV} has the following asymptotic behavior as $Z\to\infty$
\begin{equation}
         R(Z) =  I+\frac{1}{2iZ}\begin{pmatrix}
                -\fm_{V}(X,T) & \Psi_{V}(X,T) \\
                \wo{\Psi}_{V}(X,T) & \fm_{V}(X,T)
            \end{pmatrix} + \cO(Z^{-2}),
        \end{equation}
        $\Psi_{V}(X,T)$ solves the focusing NLS equation \eqref{eq:NLS}, and for $T=0$ it admits the following integral representation
    \begin{equation}
        \Psi_V(X,0) = \frac{\kappa}{X} \exp\left(-\frac{1}{2i \zeta}\int_{1}^{2i\zeta X} \frac{u(s)}{1-u(s)}ds\right)\,,
    \end{equation}
    where $u(x)$ solves the Painlev\'e-V equation \eqref{eq:PV} and $\kappa$ is independent of $X$.
\end{theorem}
\begin{remark}
We conjecture that 
\[
\mathfrak{m}_{V}(X,T) = \int_X^{+\infty} |\Psi_{V}(s,T)|^2\, \mathrm{d}s,
\]
and in future work, where we will analyze the asymptotic behavior of $\Psi_{V}$ for $X \to + \infty$, the convergence of the above integral will be fully justified. 
\end{remark}
As before, we conclude the proof of Theorem \ref{thm:main2_deterministic} by applying a small norm argument to the matrix $E(Z;X,T) = B(Z;X,T)R^{-1}(Z;X,T)$. Specifically, one can show by direct calculation that $E(Z;X,T)$ satisfies the following RHP.

\begin{RHP}
\label{rhp:E_PV}
        Let $(X,T)\in\R^2$ be two arbitrary parameters. Find a $2\times2$ matrix $E(Z;X,T)\equiv E(Z)$ such that:
        \begin{enumerate}
        \item $E(Z)$ is analytic for $|Z|\neq1$, and takes continuous boundary values $E_{\pm}$ on $|Z|=1$ from the interior and the exterior of the circle.
        \item The two boundary values are related by 
        \begin{align}
            & E_+(Z) = E_-(Z) (I +W_N(Z;X,T)), \qquad |Z|=1\\
            & W_N(Z;X,T) =\wt R_-(Z)\begin{pmatrix}
                 \cO(N^{2\delta-1}) &\cO(N^{2\delta-1})\\
                \cO(N^{2\delta-1}) & \cO(N^{2\delta-1})
            \end{pmatrix}\wt R_-(Z)^{-1}\\
            & \wt R(Z) = R(Z)\exp\left(2i\left(XZ + TZ^2 + \frac{\mu_{\cD_1}}{\zeta}\ln\left(\frac{Z+\zeta}{Z}\right) \right)\sigma_3 \right).
            \end{align}
        \item $E(Z;X,T)$ has the following expansion as $Z\to\infty$
        
        \begin{equation}
            E(Z;X,T) = I + \frac{1}{2iZ}\begin{pmatrix}
                \fm_{V}(X,T)-\fm_{N,V}(X,T) & \Psi_{V}(X,T)- \Psi_{N,V}(X,T) \\
                \wo\Psi_{V}(X,T)-\wo{\Psi}_{N,V}(X,T) &  \fm_{N,V}(X,T) -\fm_{V}(X,T)
            \end{pmatrix} + \cO(Z^{-2})\,,
        \end{equation}
    \end{enumerate}
    where we defined $\Psi_{N,V}(X,T) \coloneqq \frac{1}{N}\psi_{N,V}\left(\frac{X}{N},\frac{T}{N^2}\right)$ and analogously for $ \fm_{N,V}(X,T)$.
\end{RHP}

Then, the following Lemma, whose proof can be found in Appendix \ref{app:existence_theory}, holds.

\begin{lemma}
\label{lem:small_norm_PV}
    Consider RHP \ref{rhp:E_PV}. It has a unique solution and 
    \begin{equation}
        E(Z) = I + \cO\left(\frac{1}{Z N^{1-2\delta}}\right), \quad  Z\to\infty\,.
    \end{equation}
\end{lemma}
Applying the previous lemma, we deduce that there exists a constant $C$ such that
\begin{equation}
    |\Psi_{V}(X,T)- \Psi_{N,V}(X,T)| \leq C N^{2\delta -1}\,.
\end{equation}
Finally, since $|\Psi_{N,V}(0,0)|= 2\frac{\sum_{j=1}^N \mu_j}{N}$, we deduce that $|\Psi_{V}(0,0)| =2\mu_{\cD_1} $.
\qed

\subsubsection{Proof of Theorem \ref{thm:main2}}

We are now in position to prove Theorem \ref{thm:main2}. From Theorem \ref{thm:main2_deterministic}, we know that there exists $(N_0,\delta)\in \N\times \left(\frac{1}{4},\frac 12\right)$, such that for all $N>N_0$

    \begin{align}
        \mathbb{E}\bigg[\bigg\vert\frac{1}{N} \psi_{N,V}&\left(\frac{X}{N},\frac{T}{N^2}\right) - \Psi_{V}(X,T)\bigg\vert\bigg] \\&= \int_{\Omega_\delta\cap (\R^N \times \mathcal{U}_{2\delta})} \left\vert\frac{1}{N} \psi_{N,V}\left(\frac{X}{N},\frac{T}{N^2}\right) - \Psi_{V}(X,T)\right\vert\di \boldsymbol{v} \, \di \boldsymbol{\mu}\\ &\quad + \int_{(\Omega_\delta\cap (\R^N \times \mathcal{U}_{2\delta}))^c} \left\vert\frac{1}{N} \psi_{N,V}\left(\frac{X}{N},\frac{T}{N^2}\right) - \Psi_{V}(X,T)\right \vert\di \boldsymbol{v} \, \di \boldsymbol{\mu}\\
        &\leq CN^{2\delta-1} +  \int_{\Omega_\delta^c} \left\vert\frac{1}{N} \psi_{N,V}\left(\frac{X}{N},\frac{T}{N^2}\right) - \Psi_{V}(X,T)\right \vert\di \boldsymbol{v} \, \di \boldsymbol{\mu} \\
        &\quad +\int_{ \mathcal{U}_{2 \delta}^c \times \mathbb{R}^{N}} \left\vert\frac{1}{N} \psi_{N,V}\left(\frac{X}{N},\frac{T}{N^2}\right) - \Psi_{V}(X,T)\right \vert\di \boldsymbol{v} \, \di \boldsymbol{\mu}\\
        & \leq CN^{2\delta-1} + \int_{\Omega_\delta^c} \left\vert\frac{1}{N} \psi_{N,V}\left(\frac{X}{N},\frac{T}{N^2}\right)\right\vert\di \boldsymbol{v} \, \di \boldsymbol{\mu} + \int_{\Omega_\delta^c}\left\vert \Psi_{V}(X,T)\right \vert\di \boldsymbol{v} \, \di \boldsymbol{\mu} \\ 
        &\quad +\int_{  \mathcal{U}_{2 \delta}^c \times \mathbb{R}^{N}} \left\vert\frac{1}{N} \psi_{N,V}\left(\frac{X}{N},\frac{T}{N^2}\right)\right\vert\di \boldsymbol{\mu}\, \di \boldsymbol{v}\\ &\quad + \int_{ \cU_{2 \delta}^c \times \mathbb{R}^{N}}\left\vert \Psi_{V}(X,T)\right \vert\di \boldsymbol{v} \, \di \boldsymbol{\mu}.
    \end{align}
Since $\Psi_{V}(X,T)$ is a continuous function and we are considering $(X,T)$ in a compact set, by Heine-Borel theorem there exists a constant $C_1$ such that $|\Psi_{V}(X,T)|\leq C_1$, and there exist constants $C_2$, $\varepsilon >0$ 
    \begin{equation}
\mathbb{P}(\Omega_\delta^c)+\mathbb{P}(  \mathbb{R}^{N}\times \cU_{2 \delta}^c )\leq \exp(-C_2 N^{\varepsilon})\,,
    \end{equation}
    therefore we deduce that there exists a constant $\wt C_2$ such that 
    \begin{equation}
       \int_{   \mathbb{R}^{N}\times \cU_{2 \delta}^c}\left\vert \Psi_{V}(X,T)\right \vert\di \boldsymbol{v} \, \di \boldsymbol{\mu} + \int_{\Omega_\delta^c}\left\vert \Psi_{V}(X,T)\right \vert\di \boldsymbol{v} \, \di \boldsymbol{\mu}\leq \wt C_2 N^{-\varepsilon}\,.
    \end{equation}
Following the same steps as in Section \ref{sec:PIII}, one conclude that there exist two constants $c,\epsilon>0$ such that
    \begin{equation}
        \int_{ \Omega_\delta^c} \left\vert\frac{1}{N} \psi_{N,V}\left(\frac{X}{N},\frac{T}{N^2}\right)\right\vert\di \boldsymbol{v} \, \di \boldsymbol{\mu} \leq c N^{-\epsilon}\,.
    \end{equation}
    Regarding the last term, we notice that it can be estimated as 

    \begin{equation} 
    \int_{    \mathbb{R}^{N}\times \cU_{2 \delta}^c} \left\vert\frac{1}{N} \psi_{N,V}\left(\frac{X}{N},\frac{T}{N^2}\right)\right\vert\di \boldsymbol{v}\, \di \boldsymbol{\mu} = \int_{ \cU_{2\delta}^c} \left\vert\frac{1}{N} \psi_{N,V}\left(\frac{X}{N},\frac{T}{N^2}\right)\right\vert\di \boldsymbol{\mu} \leq \frac{2}{N} \int_{\cU_{2 \delta}^c}\sum_{j=1}^N \mu_j\di\boldsymbol{\mu}. 
    \end{equation}
We want to reduce to the same situation as in proof of Theorem \ref{thm:main1}, and therefore we express:
\begin{equation}
\cU_{2 \delta}^{c} = \left( \cU_{2 \delta}^{c} \cap V_{2 \delta}\right) \cup \left( \cU_{2 \delta}^{c} \cap V_{2 \delta}^{c}\right) 
\end{equation}
where $V_{2 \delta}$ is as in equation \eqref{eq:V_2delta}. We split the last integral as follows:
    \begin{align*}
\int_{\cU_{2 \delta}^c}\sum_{j=1}^N \mu_j\di\boldsymbol{\mu} = \int_{\cU_{2 \delta}^{c} \cap V_{2 \delta}}\sum_{j=1}^N \mu_j\di\boldsymbol{\mu} + \int_{\cU_{2 \delta}^{c} \cap V_{2 \delta}^{c}}\sum_{j=1}^N \mu_j\di\boldsymbol{\mu}.
    \end{align*}
We estimate the first integral as follows:
\begin{equation}
\Bigg| \int_{\cU_{2 \delta}^{c} \cap V_{2 \delta}}\sum_{j=1}^N \mu_j\di\boldsymbol{\mu}  \Bigg| \leq \int_{\cU_{2 \delta}^{c} \cap V_{2 \delta}} \Bigg| \sum_{j=1}^N \mu_j - N \mu_{\cD_1} \Bigg| + N \mu_{\cD_1}  \di \boldsymbol{\mu} \leq \left( N \mu_{\cD_1} + N^{2 \delta} \right) \mathbb{P} \left( \cU_{2 \delta}^{c}  \right) \leq \left( N \mu_{\cD_1} + N^{2 \delta} \right) e^{-c N^{\epsilon}}\,,
\end{equation}
where in the second inequality we used the definition of the set $V_{2 \delta}$. Regarding the second integral, we estimate it as follows,
\begin{equation}
\int_{\cU_{2 \delta}^{c} \cap V_{2 \delta}^{c}}\sum_{j=1}^N \mu_j\di\boldsymbol{\mu} \leq \int_{V_{2 \delta}^{c}}\sum_{j=1}^N \mu_j\di\boldsymbol{\mu}
\end{equation}
and in the previous section we showed that the last integral goes to 0 as $N \to \infty$ with the correct upper bound. This completes the proof of the theorem. \qed

\subsubsection{The Painlev\'e--V RHP}
In this subsection, we prove Theorem \ref{thm:rhp_pv}. First, by following the exact same arguments as in \cite{Bilman2020}, one can show that $\Psi_V(X,T)$ solves NLS. Specifically, we introduce 
\begin{equation}
\label{eq:W_def}
W(Z;X,T) \coloneqq R(Z;X,T) \, e^{-i \sigma_3 \left( XZ+TZ^2 + \frac{\mu_{\cD_1}}{\zeta} \ln \left( \frac{Z+\zeta}{Z} \right) \right)}.
\end{equation}
and one may verify that $W(Z;X,T)$ satisfies 
\begin{equation}
\label{eq:lax_pair_easy}
    W_X = A W = \begin{pmatrix}
        -iZ & \Psi_V(X,T) \\
        -\wo{\Psi_V}(X,T) & i Z
    \end{pmatrix}W, \quad W_T = B W = \begin{pmatrix}
        -iZ^2 + \frac{i}{2} |\Psi_V|^2 & Z \Psi_V+ \frac{i}{2}\Psi_{V,X} \\
        -Z\wo{\Psi_V} + \frac{i}{2}\wo{\Psi_{V,X}} & iZ^2 - \frac{i}{2} |\Psi_V|^2
    \end{pmatrix}W\, .
\end{equation}
Since the jumps of $W(Z;X,T)$ are also independent of $Z$, one can find a matrix differential equation that $W$ must satisfy as a function of $Z$. Since we are interested in the connection to the Painlev\'{e} V equation, we  focus on the case $T=0$. In this case, we want to determine explicitly the quantity
\begin{equation}
    \Lambda(Z,X) \coloneqq W_Z(Z,X,0)W^{-1}(Z,X,0)\,,
\end{equation}
which is meromorphic for $Z \in \mathbb{C}$. We prove the following:
\begin{proposition}
Define $W(Z;X,T)$ as in \eqref{eq:W_def}. Then for $T=0$, $W(Z;X,0)$ satisfies the following Lax Pair

\begin{equation}
\label{eq:lax_pair_our}
    W_Z=\Lambda W \qquad W_X=AW\,,
\end{equation}
where $A$ is as \eqref{eq:lax_pair_easy}, while $\Lambda$ has the following form

\begin{equation}
    \Lambda(Z,X) = -iX \sigma_3 + \frac{1}{Z}\left( i X \left[\sigma_3,R_1(X,0)\right] - \partial_X(X R_1(X,0)) + i \frac{\mu_{\cD_1}}{\zeta}\sigma_3\right) + \frac{1}{Z+\zeta}\left(\partial_X(XR_1(X,0)) -i\frac{\mu_{\cD_1}}{\zeta}\sigma_3\right).
\end{equation}
Here, $R_1(X,0)$ is 
\begin{equation}
    R_1(X,0) = \frac{1}{2i}  \begin{pmatrix}
            -\fm_{V}(X,0) & \Psi_{V}(X,0)\\
            \wo {\Psi}_{V}(X,0) & \fm_{V}(X,0)
        \end{pmatrix}\,.
\end{equation}
\end{proposition}
\begin{proof}
At $T=0$, $W$ simply becomes
\begin{equation}\label{eq:eqW}
    W(X,Z) = R(X,Z) \, e^{-i \sigma_3 \left( X Z + \frac{\mu_{\cD_1}}{\zeta} \ln \left( \frac{Z+\zeta}{Z} \right) \right)}.
\end{equation}
We notice that $W$ has a constant jump across the unit circle and a constant jump across the interval $(-\zeta,0)$. Here $\mu_{\cD_1}$ is the mean of the imaginary part of the discrete eigenvalues. Assuming the $Z$-asymptotic expansion for $R(X,Z)$
\[R(X,Z) = I_2 + \frac{R_1(X)}{Z} + \frac{R_2(X)}{Z^2} + \mathcal{O}(Z^{-3})\]
where we suppressed the $T$-dependence since we set $T=0$. Taking the $Z$-derivative on \eqref{eq:eqW}, we get:
\[W_{Z}(X,Z) = \left( - \frac{R_1(X)}{Z^2} - i \sigma_3 \left( X - \frac{\mu_{\cD_1}}{\zeta Z^2} \right) - iX \frac{R_1(X) \sigma_3}{Z} - i X \frac{R_2(X) \sigma_3}{Z^2} + \mathcal{O}(Z^{-3})\right) \,  e^{-i \sigma_3 \left( X Z + \frac{\mu_{\cD_1}}{\zeta} \ln \left( \frac{Z+\zeta}{Z} \right) \right)}.\]
Moreover, by equation \eqref{eq:eqW} we have that:
\begin{equation}
W^{-1}(X,Z) =  e^{i \sigma_3 \left( X Z + \frac{\mu_{\cD_1}}{\zeta} \ln \left( \frac{Z+\zeta}{Z} \right) \right)} \, \left( I_2 - \frac{R_1(X)}{Z} - \frac{R_2(X)}{Z^2} + \frac{R_1^2(X)}{Z^2} + \mathcal{O}(Z^{-3}) \right).
\end{equation}
We now define the matrix-valued function 
\[\Lambda(X,Z) = W_{Z}(X,Z) \, W^{-1}(X,Z).\]
Its large-$Z$ asymptotics is derived by combining the large-$Z$ asymptotics of $W_Z$ and $W^{-1}$. In particular, we get:
\begin{equation}\label{eq:P}
\begin{split}
      \Lambda(X,Z) =& -i \sigma_3 X + i X [\sigma_3,R_1(X)] \frac{1}{Z} \\ &+ \left( i \sigma_3 \frac{\mu_{\cD_1}}{\zeta} - R_1(X) + i X [\sigma_3, R_2(X)] + iX [R_1(X), \sigma_3 R_1(X)] \right) \frac{1}{Z^2} + \mathcal{O}(Z^{-3})
\end{split}
\end{equation}
where $[A, B] = A B - B A$. From \cite[Eq (41)]{Bilman2020}, we know that we can express the matrix $R_2(X)$ in terms of $R_1(X)$ via the following relation
\begin{equation}\label{eq:A2}
[\sigma_3,R_2(X)] = [R_1(X),\sigma_3 R_1(X)] - i R_{1X}(X).
\end{equation}
Replacing \eqref{eq:A2} into \eqref{eq:P}, we reduce the asymptotics of $\Lambda$ into an expression that only involves $R_1(X)$:
\begin{equation}\label{eq:largeZP1}
\Lambda(X,Z) = - i \sigma_3 X + i X  [\sigma_3,R_1(X)] \frac{1}{Z} + \left( i \sigma_3 \frac{\mu_{\cD_1}}{\zeta} - \partial_X \left( X R_1(X) \right) \right) \frac{1}{Z^2} + \mathcal{O}(Z^{-3}).
\end{equation}
Since $W$ has constant jumps on the unit circle, $W_Z$ has the same jumps as $W$, and therefore $\Lambda$ has no jump on the unit circle. By direct calculation, one can show that  the only points where $\Lambda$ can be singular are $Z=0,-\zeta$, where it might have a simple pole. Therefore, we can write an exact representation for $\Lambda(X,Z)$ as 
\begin{equation}\label{eq:Pv2}
    \Lambda(X,Z) = \Lambda_{an.}(X,Z) + \frac{\alpha(X)}{Z} + \frac{\beta(X)}{Z+\zeta}
\end{equation}
where $\Lambda_{an.}(X,Z)$ is a $2 \times 2$ matrix analytic everywhere in $\mathbb{C}$, $\alpha(X)$ and $\beta(X)$ are two $2 \times 2$ matrices that can depend on $X$ to be determined. Large-$Z$ asymptotics of \eqref{eq:Pv2} yields
\begin{equation}\label{eq:largeZP2}
\Lambda(X,Z) = \Lambda_{an.}(X,Z) + \frac{\alpha(X) + \beta(X)}{Z} - \frac{\beta(X)}{Z^2} + \mathcal{O}(Z^{-3}).
\end{equation}
Equating Eqs. \eqref{eq:largeZP1} and \eqref{eq:largeZP2}, we derive equations for $\alpha$ and $\beta$:
\begin{equation}\label{eq:coeffs}
\alpha + \beta = i X [\sigma_3, R_1(X)], \quad \beta = \partial_X \left( X R_1(X) \right) - i \sigma_3 \frac{\mu_{\cD_1}}{\zeta}.
\end{equation}
Moreover, we get an explicit expression for the analytic part of $\Lambda(X,Z)$:
\[\Lambda_{an.}(X) = - i \sigma_3 X\]
which concludes the proof.
\end{proof}

To complete the proof of Theorem \ref{thm:rhp_pv}, we need to show how $\Psi_V(X,0)$ and the fifth Painlev\'e transcendent are related. To do that, we consider the RHP formulation of PV \cite[Chapter 5.4]{fokas2006painleve}.
\begin{theorem}
    Consider the following Lax Pair
    
    \begin{gather}\label{eq:RHP62}
\Psi_{\lambda}(\lambda,x) = A(\lambda,x) \, \Psi(\lambda,x), \quad \Psi_{x}(\lambda,x) = U(\lambda,x) \, \Psi(\lambda,x)\\
A(\lambda,x) = \frac{x}{2} \sigma_3 + \frac{A_1(x)}{\lambda} + \frac{A_2(x)}{\lambda-1}, \quad U(\lambda,x) = \frac{\lambda}{2} \sigma_3 + B_0(x)
\end{gather}
where 
\begin{gather}
A_1(x) = \begin{pmatrix}
v + \frac{\theta_0}{2} & -y(v+\theta_0)\\ \\
\frac{v}{y} & -(v+\frac{\theta_0}{2})
\end{pmatrix}\\
A_2(x) = \begin{pmatrix}
-v-\frac{1}{2}(\theta_0 + \theta_{\infty}) & y u (v+\frac{1}{2} (\theta_0 + \theta_{\infty} - \theta_1))\\ \\
- \frac{1}{y u }(v+\frac{1}{2} (\theta_0 + \theta_{\infty} + \theta_1)) & v+\frac{1}{2}(\theta_0 + \theta_{\infty})
\end{pmatrix}\\
B_0(x) = \frac{1}{x}\begin{pmatrix}
0 & -y[v+\theta_0-u(v+\frac{1}{2}(\theta_0+\theta_{\infty}-\theta_1))]\\
\frac{1}{y}[v- \frac{1}{u}(v+\frac{1}{2}(\theta_0+\theta_{\infty}+\theta_1))] & 0
\end{pmatrix}\,.
\end{gather}
Then the function $u=u(x)$ solves the PV equation
\begin{equation}\label{eqforu}
\frac{d^2 u}{dx^2} = \left( \frac{1}{2u} \pm \frac{1}{u-1} \right) \left( \frac{du}{dx} \right)^2 - \frac{1}{x} \frac{du}{dx} + \frac{(u-1)^2}{x^2} \left( \alpha u + \frac{\beta}{u} \right) + \frac{\gamma u}{x} + \frac{\delta u(u+1)}{u-1}\,.
\end{equation}
The constants $\alpha$, $\beta$, $\gamma$ and $\delta$ are given by 
\begin{gather}
\alpha = \frac{1}{8} \left( \theta_0 - \theta_1 + \theta_{\infty} \right)^2, \quad \beta = -\frac{1}{8} \left( \theta_0 - \theta_1 - \theta_{\infty} \right)^2, \quad \gamma = 1- \theta_0 - \theta_1, \quad \delta=-\frac{1}{2}. 
\end{gather}
\end{theorem}
The following theorem completes the proof of Theorem \ref{thm:rhp_pv} and finishes the proof of  Theorem \ref{thm:main2}.
\begin{theorem}
    Consider the RHP \ref{rhp:PainleveV} for the function $R(Z,X)$; then there exists a unique constant $\kappa$  such that the function $\Psi_V(X,0)$ has  the following  integral representation

    \begin{equation}
        \Psi_V(X,0)=\frac{\kappa}{X} \exp\left( -\frac{1}{\zeta} \int_{1}^{2i\zeta X} \frac{u(s)}{1-u(s)}\di s\right)\,,
    \end{equation}
    where $u(s)$ solves \eqref{eqforu} with $\theta_\infty=0,\theta_0=-\theta_1=2i\frac{\mu_{\cD_1}}{\zeta}$.
\end{theorem}
\begin{remark}
    We notice that 
    
    \begin{equation}
        \lim_{\zeta\to0} \Lambda(X,Z) = -iX\sigma_3 + \frac{1}{Z}iX[\sigma_3,R_1] + i\frac{\mu_{\cD_1}}{Z^2}=:\Lambda_0(X,Z)\,.
    \end{equation}
    Therefore, in this regime the Lax pair reads

    \begin{equation}
        W_X = AW\,,\qquad W_Z=\Lambda_0W\,,
    \end{equation}
    which, up to a rescaling, is exactly the Lax Pair of the Painlev\'e-III Hierarchy of Sakka \cite{Sakka2009}, see also \cite{Bilman2020}.
\end{remark}

\begin{proof}
Notice that under the change of variables 
\begin{equation}
\label{eq:change_to_PV}
    x = 2 i \zeta X, \quad \lambda = -\frac{1}{\zeta} Z
\end{equation}
the Lax pairs ~\eqref{eq:lax_pair_our} and ~\eqref{eq:RHP62} become equivalent, i.e., $W(X,Z) \equiv \Psi \left(2 i \zeta X, -\frac{1}{\zeta} Z\right)$, where $W$ solves RHP~\ref{eq:lax_pair_easy} and $\Psi$ solves RHP~\ref{eq:RHP62}, respectively. In this proof we compute explicitly the constants $\theta_{\infty}$, $\theta_{0}$ and $\theta_{1}$, and express the functions $u$, $v$ and $y$ appearing in the Lax pair \eqref{eq:RHP62} in terms of the potential $\psi(X,Z)$. Specifically, using the large-$Z$ asymptotics of $W$, and its behavior nearby $Z=0$ and $Z=-\zeta$, we deduce the constants $\theta_{\infty}$, $\theta_{0}$ and $\theta_{1}$. First, recall the definition of $W$, \eqref{eq:eqW}, which yields the large-$Z$ asymptotics for $W$ 
\begin{equation}\label{eq:asymW}
W(X,Z) = \left( I_2 + \mathcal{O}(Z^{-1}) \right) e^{-i X Z \sigma_3} \left( \frac{Z+\zeta}{Z} \right)^{\frac{-i \mu_{\cD_1}}{\zeta} \sigma_3}.
\end{equation}
Using the large-$Z$ asymptotics of $\Psi$, which we draw from \cite{fokas2006painleve}, after re-writing it into the $(X,Z)$ variables, we get
\begin{equation}\label{eq:asympsi}
 \Psi \left(2 i \zeta X, -\frac{1}{\zeta} Z\right) = \left( I_2 + \mathcal{O}(\zeta Z^{-1}) \right) e^{-i X Z \sigma_3} \left( - \frac{ \zeta}{Z} \right)^{\frac{\theta_{\infty}}{2} \sigma_3}.
\end{equation}
Comparing Eqs.~\eqref{eq:asymW} and \eqref{eq:asympsi}, we deduce that $\theta_{\infty}=0$. Additionally, nearby $Z=0$, we can express $W$ as follows
\begin{equation}\label{eq:eq618}
W(X,Z) = W^{0}(X,Z) Z^{\frac{i \mu_{\cD_1}}{\zeta} \sigma_3} \equiv R(X,Z) e^{-i X Z \sigma_3} \left( Z + \zeta \right)^{\frac{-i \mu_{\cD_1}}{\zeta} \sigma_3}
\end{equation}
where $W^{0}(X,Z)$ is an analytic function around $Z=0$. Using the behavior of $\Psi$ nearby $Z=0$, which we draw from \cite{fokas2006painleve}, after re-writing it into the $(X,Z)$ variables
\begin{equation}\label{eq:eq619}
\Psi \left(2 i \zeta X, -\frac{1}{\zeta} Z\right) = \Psi^{0} \left(2 i \zeta X, -\frac{1}{\zeta} Z\right) Z^{\frac{\theta_0}{2} \sigma_3}
\end{equation}
where $\Psi^{0} \left(2 i \zeta X, -\frac{1}{\zeta} Z\right)$ is an analytic function around $Z=0$. Comparing Eqs. \eqref{eq:eq618} and \eqref{eq:eq619}, we deduce that $\theta_0=\frac{2i \mu_{\cD_1}}{\zeta}$. A similar analysis around $Z=-\zeta$ gives that $\theta_1=-\frac{2i \mu_{\cD_1}}{\zeta}$.

Furthermore, by equating the two equivalent forms of the Lax pair, we deduce that

\begin{equation}
    u(2i\zeta X) = (1-2i(\partial_X\ln(X\Psi_V(X,0))^{-1})^{-1}\,,
\end{equation}
inverting this relation, we deduce that there exists a constant $\kappa$ such that

\begin{equation}
    \Psi_V(X,0) = \frac{\kappa}{X}\exp\left(-\frac{1}{2i\zeta}\int_{1}^{2i\zeta X} \frac{u(s)}{1-u(s)}ds\right)
\end{equation}
which is equivalent to \eqref{eq:psiV_as_PV}.

\end{proof}
\section*{Acknowledgments}
This research was supported by the Louisiana Board of Regents Endowed Chairs for Eminent Scholars program.
\appendix

\section{Existence theory}
\label{app:existence_theory}

\subsection{Proof of the small-norm argument}
In this section, we prove Lemma \ref{lem:small_norm_PIII} and Lemma \ref{lem:small_norm_PV}. Specifically, we prove a more general result from which the two Lemmas follow. To do that, we need some definition. First, we define the following scalar product for matrix valued functions. 

\begin{definition}
\label{def:scalar_product}
    Fix $n\in\N$, let $\Sigma\subseteq\C$ be a rectifiable curve and $M,Q\, : \, \Sigma \to \C^{n\times n}$; then we define the scalar product $\langle\cdot,\cdot\rangle$ for such matrix valued functions as

    \begin{equation}
        \langle M, Q\rangle = \int_\Sigma \trace{M^\star Q}|\di z|\,,
    \end{equation}
where $|\di z|$ is the arc-length parameter and $M^\star$ is the conjugate-transpose of the matrix $M$.
\end{definition}
The previous scalar product naturally induces a norm as
\begin{equation}
\label{eq:def_norm}
    \norm{M}_2^2=\langle M, M\rangle\,.
\end{equation}
Another norm for matrix-valued function $M:\, \Sigma \,\to \C^{n\times n}$ of interest is the $\infty$-norm defined as
\begin{equation}
    \norm{M}_\infty = n \essesup_{z\in\Sigma} \max_{i,j}|M_{i,j}(z)|\,.
\end{equation}
We notice that for two matrix-valued functions $M,Q:\, \Sigma \,\to \C^{n\times n}$ the following holds.
\begin{equation}
\label{eq:ineq_norm}
    \norm{MQ}_2\leq \norm{M}_\infty\norm{Q}_2\,.
\end{equation}
Naturally, the $\norm{\cdot}_2$ norm induces an operator norm as follows.
\begin{definition}
\label{def:op_norm}
    Fix $n\in\N$, let $\Sigma\subseteq\C$ be a rectifiable curve. Let $C\, : \text{Mat}(n,\Sigma) \to \text{Mat}(n,\Sigma)$, where

    \begin{equation}
        \text{Mat}(n,\Sigma) \coloneqq \{M:\, \Sigma \,\to \C^{n\times n} \vert \norm{M}_2 < \infty \}\,.
    \end{equation}
    We define the operator norm $\vertiii{C}_2$ as

    \begin{equation}
        \vertiii{C}_2 \coloneqq \sup_{\genfrac{}{}{0pt}{}{M\in \text{Mat}(n,\Sigma)}{ ||M||_2=1}} \norm{C(M)}_2\,. 
    \end{equation}
\end{definition}
Finally, we define the two Cauchy integral operators $C_+,C_-\,:\text{Mat}(2,\Sigma) \to \text{Mat}(2,\Sigma)$ as 
\begin{equation}
\label{eq:cauchy_operators}
  C_+(h)(z) \coloneqq \lim_{\genfrac{}{}{0pt}{}{\zeta\to z}{\text{ from inside of }\Sigma}} \frac{1}{2\pi i} \int_\Sigma \frac{h(s)}{s-\zeta}\di s\,,\qquad 
    C_-(h)(z) \coloneqq \lim_{\genfrac{}{}{0pt}{}{\zeta\to z}{\text{ from outside of }\Sigma}} \frac{1}{2\pi i} \int_\Sigma \frac{h(s)}{s-\zeta}\di s\,,
\end{equation}
and, for any matrix valued function $M(z)$, the operator $C_{M}\,:\text{Mat}(2,\Sigma)\to \text{Mat}(2,\Sigma)$ as
\begin{equation}
\label{eq:def_CWN}
    C_{M}(h(z))(z) \coloneqq C_-(h(z)M(z))\,.
\end{equation}
We notice that $\mathbbm{1}=C_+-C_-$.

Given these definitions, consider the following RHP for a matrix valued function $E(z)$
\begin{RHP}\label{e:RHP0}
Find a $2 \times 2$ matrix-valued function $E(z)$ such that
\begin{itemize}
	\item $E(z)$ is analytic in  $ \mathbb{C} \setminus \Sigma $.
    \item For $z \in \Sigma$, $E$ has boundary values $E_{+}(z)$ and $E_{-}(z)$, which satisfy the jump relation 
        \begin{equation}
           E_+(z)=\  E_-(z) \  (I + W_N(z)), \quad z \in \Sigma
        \end{equation}
where $\Sigma$ is a closed contour. Moreover, the $L_\infty$ norm of $W_N$ is small, i.e., there exists constant $\kappa_{1}>0$ and $\delta < \frac{1}{2}$ so that:
\begin{equation}\label{e:infnorm}
 \norm{W_N}_{\infty} \leq \kappa_1 N^{2 \delta - 1}, \quad 2 \delta < 1.   
\end{equation}
\item $E(z) = \  I + O(z^{-1})$ as $z \to \infty$.
\end{itemize}
\end{RHP}
\begin{remark}
    We notice that the two RHPs considered in Lemma \ref{lem:small_norm_PIII} and Lemma \ref{lem:small_norm_PV} have the same properties as RHP \ref{e:RHP0}.
\end{remark}
Then, the following Proposition holds
\begin{proposition}
\label{prop:small_norm_app}
RHP \ref{e:RHP0} has a unique solution given in the form
\begin{equation}\label{e:solution}
E(z) = I + \frac{1}{2 \pi i} \sum_{\ell =0}^{\infty} \int_{\Sigma} \frac{C_{W_N}^{\ell}(I)(s) \ W_N(s)}{s-z} \, ds.
\end{equation}
\end{proposition}

\begin{proof}
We articulate the proof in several steps. First, we show that, if the solution exists, it is unique. Then we show that $E_{-}(z)$ is well defined and it can be expressed as the Neumann series, which implies that $E(z)$ exists and it has the required asymptotic behavior at infinity. 

\paragraph{\textbf{Uniqueness}.}
Assume that there exists two solutions $E_1(z),E_2(z)$ to RHP \ref{e:RHP0}. Then the matrix $N(z)=E_1(z)E_2(z)^{-1}$ is analytic on the whole complex plane and has the following asymptotic behavior

\begin{equation}
    N(z)=I + \cO(z^{-1}) \quad \text{as } z\to\infty\,.
\end{equation}
Therefore, by Liouville's theorem, $N(z)\equiv I$, which implies that $E_1(z)=E_2(z)$.

\paragraph{\textbf{Existence}.}
Rewrite the jump condition of $E(z)$ as follows:
\begin{align*}
 E_+(z) &= \  E_-(z) \  (I + W_N(z))\\
 &= \  E_-(z) \  (I + W_N(z)) +  E_-(z) -  E_-(z)
\end{align*}
which is equivalent to
\[  E_+(z) -  E_-(z) = \  E_-(z) \ W_N(z).\]
Using the Sokhotski–-Plemelj formula, we derive that:
\begin{equation}\label{e:sol_E}
E(z) = I + \frac{1}{2 \pi i} \int_{\Sigma} \frac{E_-(s) \ W_N(s)}{s-z} \, ds
\end{equation}
 Taking the minus boundary value in equation \eqref{e:sol_E}, we derive that:
\[E_{-}(z) = I + C_{-}(W_N E_{-}) = I + C_{W_N}(E_{-})(z)\]
which is equivalent to 
\begin{equation}\label{e:rel}
    (\mathbbm{1} - C_{W_N}) E_-(z) = I.
\end{equation}
Next, we show that the operator $\mathbbm{1} - C_{W_N}$ is invertible by showing that its norm $ \vertiii{C_{W_N}}_{2}<1$. Consider:
\[\norm{C_{W_N} E}_2 = \norm{C_{-}(W_N E)}_2 \leq \kappa \norm{W_N E}_2 \leq \kappa \norm{W_N}_{\infty} \norm{E}_2\]
where we used the fact that $C_{W_N}$ is a bounded operator in the $L_2(\Sigma)$ \cite{DeiftKriecherbauerMcLaughlinVenakidesZhou1999}. By the definition of the operator norm, we deduce that:
\[\vertiii{C_{W_N}}_{2} \leq \kappa \norm{W_N}_{\infty} .\]
In addition, equation \eqref{e:infnorm} holds for $W_N$, which shows that $\vertiii{C_{W_N}}_{2}<1$; so the operator $\mathbbm{1} - C_{W_N}$ is invertible. Equation \eqref{e:rel} then implies that $E_{-}(z)$ exists, and it can be expressed as convergent Neumann series:
\[E_{-}(z) = \sum_{\ell =0}^{\infty} C_{W_N}^{\ell}(I)(z).\]
Furthermore, equation \eqref{e:sol_E} shows that $E(z)$ exists and can be written as in equation \eqref{e:solution}. Lastly, using \eqref{e:sol_E}, the large-$z$ asymptotics of $E(z)$ is:
\[E(z) = I - \frac{1}{2 \pi i} \sum_{j=0}^{\infty} \frac{1}{z^{j+1}} \int_{\Sigma} W_N(s) E_{-}(s) s^j \, ds\]
where 
\begin{align*}
\left| \int_{\Sigma} W_N(s) E_{-}(s) s^j \, ds \right| & \leq \int_{\Sigma} \left| W_N(s) E_{-}(s) \right| \, ds \\
& \leq \int_{\Sigma} \left| W_N(s) s^{j} \right| |E_{-}(s)| \, ds\\
& \leq \int_{\Sigma} \norm{W_N(s) s^{j}}_{\infty} |E_{-}(s)| \, ds \\
& = \norm{W_N(s) s^{j}}_{\infty} \int_{\Sigma}|E_{-}(s)| \, ds \\
& \leq \ell^{1/2}(\Sigma) \norm{W_N(s) s^{j}}_{\infty} \norm{E_{-}(s)}_2 
\end{align*} 
(where we used that $\Sigma$ is a finite length contour) and there exists two constants $C_1,C_2$ such that $ \norm{W_N(s) s^{j}}_{\infty}\leq C_1 \, N^{2\delta-1}$  and $\norm{E_{-}(s)}_2\leq C_2 $. This implies that $E(z)$ has well-defined asymptotic behavior at infinity. 

Lastly, we need to prove that $E$ has continuous boundary values as $z$ approaches the circle (from either side).  To do so, we use the analyticity of the jump matrix $I + W_{N}(z)$ to define $\tilde{E}(z)$.  Choose a slightly larger circle $\{z: |z| = 1+\tilde{\epsilon} \}$ and let $\tilde{\mathcal{A}}$ be the annular region between the unit circle and this new circle.  Then define
\begin{equation}
    \tilde{E} = E(z)\begin{cases} I + W_{N}(z), & z \in \tilde{\mathcal{A}}\\
    I , & \mbox{ otherwise.} \\
    \end{cases}
\end{equation}
The reader may verify that $\tilde{E}$ solves a RHP which is of the same form as that of $E$, except that the contour is $\{|z| = 1+\tilde{\epsilon} \}$, and the jump matrix is the analytic continuation of $I + W_{N}(z)$ to the new contour, where the same uniform bounds still hold. Now the previous $L^{2}$ existence theory holds for $\tilde{E}$, and in particular we learn that $\tilde{E}$ is analytic on the unit circle.  From the definition of $\tilde{E}$, we then see that for $|z|=1$, $E_{\pm}(z)$ are the boundary values (from inside or outside, respectively) of a function that is analytic in a neighborhood of the unit circle, which implies that $E(z)$ achieves its boundary values in the sense of continuous functions.
\end{proof}


\subsection{Proof of Theorem \ref{thm:existence}}

In this section, we prove Theorem \ref{thm:existence}, for the sake of the reader, we rewrite the statement.

\begin{theorem}
\label{thm:existence_app}
    Let $0<\zeta<1,\, \mu_{\cD_1}\in \R_+$ and fix $(X,T)\in\R^2$. Consider the following RHP for a matrix valued function $R(Z;X,T)\equiv R(Z)$.
    \begin{RHP}\label{e:RHPforR}
Find a $2\times2$ matrix $R(Z;X,T)$ such that

    \begin{enumerate}
        \item $R(Z;X,T)$ is analytic for $|Z|\neq1$, and takes continuous boundary values in the interior and in the exterior of the circle
        \item The two boundary values are related by 
        \begin{equation}
            R_+(Z) = R_-(Z) W(Z), \quad W(Z)=
            e^{-i\left(XZ+TZ^2+\frac{\mu_{\cD_1}}{\zeta} \ln\left(\frac{Z+\zeta}{Z}\right)\right)\sigma_3}\begin{pmatrix}
                \frac{1}{\sqrt{2}} & \frac{1}{\sqrt{2}}\\
                -\frac{1}{\sqrt{2}} & \frac{1}{\sqrt{2}}
            \end{pmatrix}e^{i\left(XZ+TZ^2+\frac{\mu_{\cD_1}}{\zeta} \ln\left(\frac{Z+\zeta}{Z}\right)\right)\sigma_3}
        \end{equation}
    for $Z$ such that $|Z|=1$.
        \item $R(Z;X,T)= I +  \cO(Z^{-1})$, as $Z\to\infty$.
    \end{enumerate}
\end{RHP}
This RHP has a unique solution.
\end{theorem}

To prove this result, we recall some useful definitions. The proof will follow the proof of Theorem 5.3 in \cite{DeiftKriecherbauerMcLaughlinVenakidesZhou1999}.

\begin{Definition}\label{e:def2}
Let $\Sigma\subseteq\C$ be a rectifiable contour, and consider an  operator $K \colon L^2(\Sigma) \to L^2(\Sigma)$. Then, $K$ is compact in the $L^2$ sense, if for any weakly convergent sequence of functions $f_n\rightharpoonup f$, then $\norm{K (f_n-f)}_{L^2} \to 0$.
\end{Definition}

\begin{Definition}\label{def:compact}
    A bounded operator $T: X \to Y$  between Banach spaces $X$ and $Y$ is Fredholm if and only if it is invertible modulo compact operators, i.e., if there exists a bounded linear operator $S:Y \to X$ such that $I-ST$ and $I-TS$ are compact operators on $Y$ and $X$, respectively. 
\end{Definition}

\begin{proof}[Proof of Theorem \ref{thm:existence_app}]

The proof of the uniqueness follows the same line as in Proposition \ref{prop:small_norm_app}.

To prove existence, we rewrite the jump condition of $R(Z)$ as follows:
\begin{align*}
 R_+(Z) &= \  R_-(Z) \  W(Z)\\
 &= \  R_-(Z) \  W(Z) +  R_-(Z) -  R_-(Z)
\end{align*}
which is equivalent to
\[ R_+(Z) -  R_-(Z) = \  R_-(z) \ \left( W(Z) - I \right).\]
Using the Sokhotski–-Plemelj formula, we derive that:
\begin{equation}
R(Z) = I + \frac{1}{2 \pi i} \int_{|s|=1} \frac{R_-(s) \  \left( W(s) - I \right)}{s-Z} \, ds
\end{equation}
Taking the minus boundary value in the previous equation, we deduce that:
\begin{align*}
R_-(Z) = I + C_-((W-I) E) = I + C_{W}(R_-)(Z)
\end{align*}
which is equivalent to
\begin{equation}
    (\mathbbm{1} - C_{W}) R_-(Z) = I.
\end{equation}
Therefore, showing that RHP \ref{e:RHPforR} has a solution is equivalent to show that $\mathbbm{1} - C_{W}$  is a bijection in $L^2(\Sigma)$. To prove that, we use Fredholm theory, specifically, we prove the following
\begin{enumerate}
    \item  $\mathbbm{1} - C_{W}$ is a Fredholm operator;
    \item  $\mathbbm{1} - C_{W}$ has index $0$;
    \item  $\Ker(\mathbbm{1} - C_{W})=\{0\}$.
\end{enumerate}

\paragraph{$\mathbbm{1} - C_{W}$ is a Fredholm operator.}
We show that the operator $\mathbbm{1} - C_{W}$ by applying Definition \ref{def:compact} directly.
This is equivalent to showing that there exists an operator $S$ such that 
\[\mathbbm{1}-S(\mathbbm{1} - C_{W}), \quad \mathbbm{1} -(\mathbbm{1} - C_{W})S\]
are compacts. This operator $S$ is the pseudo-inverse operator of $\mathbbm{1} - C_{W}$. Define $C_{W^{-1}}$ as follows:
\[ C_{W^{-1}} (h)(Z) \coloneqq  C_-((W-I)^{-1} h)\,.\]
We claim that $\mathbbm{1} - C_{{W}^{-1}}$ is the pseudo-inverse of $\mathbbm{1} - C_{W}$.
Consider
\begin{equation}
\label{eq:composition_operator}
\begin{split}
    \left( \mathbbm{1} - C_{W} \right) \left( \mathbbm{1} - C_{{W}^{-1}} \right) f &= f -  C_{{W}^{-1}} f - C_{W} f +  C_{W} [ C_{{W}^{-1}} f] \\
&= f -  C_{{W}^{-1}} f - C_{W} f + C_+ \left[ C_{{W}^{-1}} f(I-W^{-1}) \right] \\
&= f -  C_{{W}^{-1}} f - C_{W} f + C_{+} \left[ C_+\left[ f(I-W)\right] (I-W^{-1})\right]
\end{split}
\end{equation}
and using the property
\[C_+ \left[ f(I-W)\right] = C_- \left[ f(I-W)\right] + f(I-W)\]
the last equality in \eqref{eq:composition_operator} becomes:
\begin{align*}
\left( \mathbbm{1} - C_{W} \right) \left( \mathbbm{1} - C_{{W}^{-1}} \right) f &= f - C_{W} f - C_{{W}^{-1}} f + C_{+} \left[ \left\{ C_{-}\left[ f(I-W)\right] + f(I-W) \right\} (I-W^{-1})\right] \\
&= f + C_{+} \left[ -f(I-W) -f(I-W^{-1}) + \{ C_{-}[f(I-W)] + f(I-W)\} (I-W^{-1})\right] \\
&= f + C_{+} \left[ C_{-}(f \hat{v}) v\right]
\end{align*}
where we defined: $v=I-W^{-1}$ and $\hat{v} = I-W$. Next, we prove that the operator $K \colon f \mapsto C_{+} \left[ C_{-}(f \hat{v}) v\right]$ is compact, that is $K$ satisfies definition \ref{def:compact}, which is equivalent to show that for any sequence of functions $f_n\rightharpoonup 0$ then $Kf_n\to0$. By direct calculation, we obtain that: 
\begin{align*}
K f_n &= \lim_{\genfrac{}{}{0pt}{}{\xi \to z}{\text{ from the + side}}} \frac{1}{2 \pi i} \int_{\Sigma} \frac{C_{-}(f_n \hat{v})(s) v(s)}{s-\xi} \, ds\\
&= \lim_{\genfrac{}{}{0pt}{}{\xi \to z}{\text{ from the + side}}} \frac{1}{2 \pi i} \int_{\Sigma^{-}} \frac{C_{-} (f_n \hat{v})(s) v(s)}{s-\xi} \, ds\\
&= \frac{1}{2 \pi i} \int_{\Sigma^{-}} \frac{C(f_n \hat{v})(s) v(s)}{s-z} \, ds\\
& \leq \frac{1}{2 \pi} \int_{\Sigma^{-}} \left| \frac{C(f_n \hat{v})(s) v(s)}{s-z} \right| \, ds\\
& \leq \frac{1}{2 \pi} \norm{C(f_n \hat{v})}_2 \left\lVert \frac{v}{s - z} \right\rVert_2
\end{align*}
where $s \in \Sigma^{-}\coloneqq \{Z\in\C \vert |Z|=1+\delta\}$ for some $\delta>0$ small and $z \in \Sigma$. Therefore, we deduce that:
\begin{align*}
K f_n & \leq \frac{1}{2 \pi \delta^{1/2}} \norm{C(f_n \hat{v})}_2 \norm{v}_2\,,
\end{align*}
Since $f_n\rightharpoonup 0$ and the Cauchy operator is bounded in $L^2(\Sigma)$ \cite{DeiftKriecherbauerMcLaughlinVenakidesZhou1999}, we have that $\norm{C(f_n \hat{v})}_2\xrightarrow{n\to\infty}0$, so we conclude that $\mathbbm{1} - C_{W}$ is a Fredholm operator.

\begin{figure}
\centering

\tikzset{every picture/.style={line width=0.75pt}} 

\begin{tikzpicture}[x=0.75pt,y=0.75pt,yscale=-1,xscale=1]

\draw  [draw opacity=0][fill={rgb, 255:red, 155; green, 155; blue, 155 }  ,fill opacity=0.31 ,even odd rule] (272.6,161.06) .. controls (272.6,118.9) and (307,84.72) .. (349.44,84.72) .. controls (391.87,84.72) and (426.27,118.9) .. (426.27,161.06) .. controls (426.27,203.22) and (391.87,237.4) .. (349.44,237.4) .. controls (307,237.4) and (272.6,203.22) .. (272.6,161.06)(230.67,161.06) .. controls (230.67,95.74) and (283.84,42.79) .. (349.44,42.79) .. controls (415.03,42.79) and (468.21,95.74) .. (468.21,161.06) .. controls (468.21,226.38) and (415.03,279.33) .. (349.44,279.33) .. controls (283.84,279.33) and (230.67,226.38) .. (230.67,161.06) ;
\draw    (350.17,310.17) -- (349.34,2.67) ;
\draw [shift={(349.33,0.67)}, rotate = 89.85] [color={rgb, 255:red, 0; green, 0; blue, 0 }  ][line width=0.75]    (10.93,-3.29) .. controls (6.95,-1.4) and (3.31,-0.3) .. (0,0) .. controls (3.31,0.3) and (6.95,1.4) .. (10.93,3.29)   ;
\draw    (154.5,160.5) -- (543.83,159.84) ;
\draw [shift={(545.83,159.83)}, rotate = 179.9] [color={rgb, 255:red, 0; green, 0; blue, 0 }  ][line width=0.75]    (10.93,-3.29) .. controls (6.95,-1.4) and (3.31,-0.3) .. (0,0) .. controls (3.31,0.3) and (6.95,1.4) .. (10.93,3.29)   ;
\draw  [color={rgb, 255:red, 74; green, 144; blue, 226 }  ,draw opacity=1 ][line width=2.25]  (273.33,161.06) .. controls (273.33,119.04) and (307.53,84.97) .. (349.72,84.97) .. controls (391.91,84.97) and (426.11,119.04) .. (426.11,161.06) .. controls (426.11,203.09) and (391.91,237.15) .. (349.72,237.15) .. controls (307.53,237.15) and (273.33,203.09) .. (273.33,161.06) -- cycle ;
\draw  [color={rgb, 255:red, 208; green, 2; blue, 27 }  ,draw opacity=1 ][line width=2.25]  (230.67,160.64) .. controls (230.67,95.3) and (283.84,42.33) .. (349.44,42.33) .. controls (415.03,42.33) and (468.21,95.3) .. (468.21,160.64) .. controls (468.21,225.98) and (415.03,278.95) .. (349.44,278.95) .. controls (283.84,278.95) and (230.67,225.98) .. (230.67,160.64) -- cycle ;
\draw  [color={rgb, 255:red, 208; green, 2; blue, 27 }  ,draw opacity=1 ][fill={rgb, 255:red, 208; green, 2; blue, 27 }  ,fill opacity=1 ] (239.69,160.4) -- (232.6,180.59) -- (223.51,161.23) ;
\draw  [color={rgb, 255:red, 208; green, 2; blue, 27 }  ,draw opacity=1 ][fill={rgb, 255:red, 208; green, 2; blue, 27 }  ,fill opacity=1 ] (460.43,160.03) -- (466.89,139.63) -- (476.58,158.71) ;
\draw  [color={rgb, 255:red, 74; green, 144; blue, 226 }  ,draw opacity=1 ][fill={rgb, 255:red, 74; green, 144; blue, 226 }  ,fill opacity=1 ] (419.22,160.28) -- (424.66,143.04) -- (433.03,159.06) ;
\draw  [color={rgb, 255:red, 74; green, 144; blue, 226 }  ,draw opacity=1 ][fill={rgb, 255:red, 74; green, 144; blue, 226 }  ,fill opacity=1 ] (280.92,160.44) -- (274.4,177.3) -- (267.05,160.78) ;

\draw (383.33,107.4) node [anchor=north west][inner sep=0.75pt]  [font=\Large]  {$\Sigma $};
\draw (451.33,65.4) node [anchor=north west][inner sep=0.75pt]  [font=\Large]  {$\Sigma ^{-}$};
\draw (242,124.73) node [anchor=north west][inner sep=0.75pt]  [font=\Large]  {$\mathfrak{D}$};
\draw (545.83,163.23) node [anchor=north] [inner sep=0.75pt]  [font=\Large]  {$\Re ( Z)$};
\draw (351.33,4.07) node [anchor=north west][inner sep=0.75pt]  [font=\Large]  {$\Im ( Z)$};

\end{tikzpicture}

\caption{Contours $\Sigma$, $\Sigma^{-}$ and domain $\mathfrak{D}$}
\label{fig:figure2}
\end{figure}
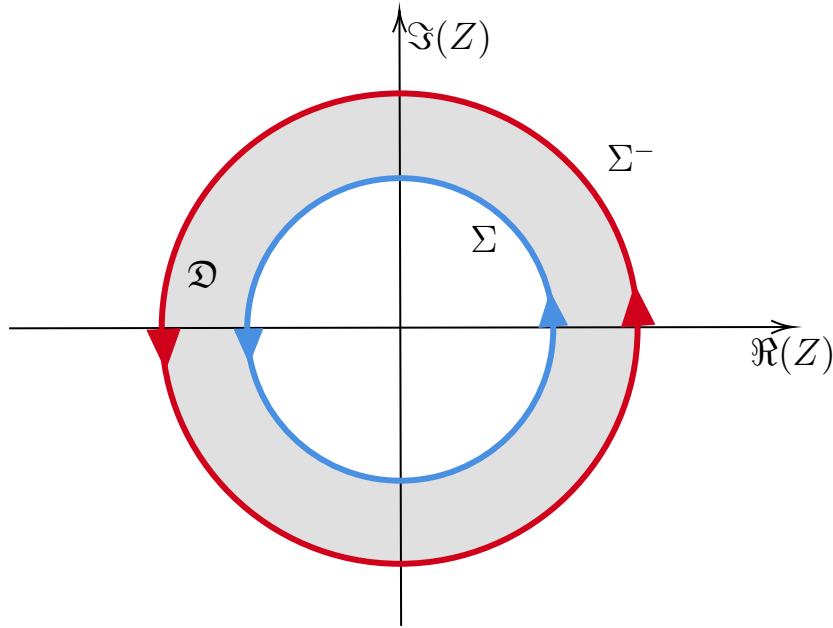

\paragraph{The operator $\mathbbm{1} - C_{W}$ has index 0.}

 Consider the family of operators $T_t \coloneqq \mathbbm{1} - C_{W}^{(t)}$ with $t \in [0,1]$ and $C_{W}^{(t)}$ defined as follows:
\[C_{W}^{(t)} f \coloneqq C_{+} \left[ f(t v)\right], \quad f \in L^2(\Sigma).\]
The map $O(t) \colon t \mapsto \mathbbm{1} - C_{W}^{(t)}$ is a continuous map from $[0,1]$ to the space of bounded operators in $L^2(\Sigma)$. Moreover, $\mathbbm{1} - C_{W}^{(t)}$ is a Fredholm operator which we can verify by repeating the above argument after replacing $C_{W}$ with $C_{W}^{(t)}$. Since the index of a operator is invariant for continuous perturbation and that \(T_0 = I\) has index \(0\) we conclude that \(\operatorname{ind}(T_t) = 0\) for all \(t \in [0,1]\).

\paragraph{Proof that $\Ker \left( \mathbbm{1} - C_{W} \right) = \{ 0 \}$.}

Suppose, to the contrary, that for some $\mu_o \in L_2(\Sigma)$ we have:
\begin{equation}\label{e:kernel}
\left( \mathbbm{1} - C_{W} \right) \mu_o = 0\,.
\end{equation}
Consider the function
\begin{equation}\label{e:funRo}
R_o(Z) = \frac{1}{2 \pi i} \int_{\Sigma} \frac{\mu_o(s)(W(s)-I)}{s-Z} \, ds, \quad Z \in \mathbb{C} \setminus \Sigma.
\end{equation}
The function $R_o(Z)$ satisfies the following RHP:
\begin{RHP}
Find a matrix value function $R_o(Z)$ such that
\begin{itemize}
\item[1.] $R_o(Z)$ is analytic in  $ \mathbb{C} \setminus \Sigma $
\item[2.] $\displaystyle  R_o(Z) = \cO(Z^{-1}), \quad Z\to\infty$
\item[3.] $\displaystyle R_{o,+}(Z) = R_{o,-}(Z) W(Z), \quad Z \in \Sigma.$
\end{itemize}
\end{RHP}
The first two conditions follow immediately from the definition of $R_o(Z)$. The third one derives from the follow chain of equalities for $s \in \Sigma$
\begin{align*}
R_{o,+}(s) &= C_{+} \left( \mu_o(s) (W(s) - I) \right)\\
&= (I-C_{-})\left( \mu_o(s) (W(s) - I) \right)\\
&= \mu_o(s) W(s) - \mu_o(s) + C_W(\mu_o)(s)
\end{align*}
where we used the fact that $C_+=I-C_{-}$ for the Cauchy operator \eqref{eq:cauchy_operators}. Now, using relation \eqref{e:kernel}, we can recast the last equality as
\[R_{o,+}(s) =  \mu_o(s) W(s).\]
Moreover, using the definition of $R_o(Z)$:
\[R_{o,-}(Z) =  C_{-} \left( \mu_o(s) (W(s) - I) \right) \stackrel{\eqref{eq:def_CWN}}{=} C_W(\mu_o).\]
Then, by comparing the two boundary values, we get the jump relation
\[R_{o,+}(Z) = R_{o,-}(Z) W(Z), \quad Z \in \Sigma.\]
By assumption, we know only that $\mu_o \in L_2(\Sigma)$, and so this jump relation must be interpreted in the $L^2$ sense as follows:
\begin{equation}\label{e:l2_convergence}
\lim_{\epsilon \downarrow 0} \int_{\Sigma} \left| R_o(s\pm i \epsilon) - R_o(s) \right|^2 |d s| = 0\,.
\end{equation}
 Next, we show that the boundary values of $R_o(Z)$ exist as continuous functions. To do that, we first define the function $R_1(Z)$ as follows:
\begin{equation}
R_1(Z) = \begin{cases}
R_o(Z) \, W(Z), &\quad Z \in \mathfrak{D}\\
R_o(Z), &\quad Z \in \mathbb{C}\setminus \mathfrak{D}
\end{cases}
\end{equation}
where $\mathfrak{D}$ is the domain between the two circles shown in Figure \ref{fig:figure2}, and the two circles have counterclockwise orientation. Notice that for $Z \in \Sigma$, $R_1(Z)$ has no jump:


\[R_{1,+}(Z) = R_{o,+}(Z) = R_{o,-}(Z) W(Z) = R_{1,-}(Z)\]
and $R_1(Z)$ has a jump on $\Sigma^{-}$, since for $Z\in \Sigma^{-}$:
\[R_{1,+}(Z) = R_{o,+}(Z) W(Z) = R_{o}(Z) W(Z) = R_{1,-}(Z) W(Z).\]
Next, we show that $R_1(Z)$ is analytic in a small neighborhood around $\Sigma$. The proof relies on Morera's theorem, for which we need to prove that for any closed smooth contour contained in this small neighborhood around $\Sigma$, the integral of $R_{1}$ is $0$. Therefore, consider any such contour, then by deformation one can reduce it to a contour of the form shown in Figure \ref{fig:analyticity}. If all integrals over such closed contours are $0$, then it is a standard result that any closed smooth contour which crosses $\Sigma$ will also integrate to $0$.  So we restrict our attention to these contours. 
Let $\epsilon$ represent the distance from these contours to the unit circle.  Then we have
\begin{align}&
\int_{\gamma_{2}} R_{1}(s) ds + \int_{\gamma_4} R_{1}(s) ds = - \int_\alpha^\beta \left[ R_o \left( (1-\epsilon) e^{i \theta}\right) - R_o \left( (1+\epsilon) e^{i \theta}\right)   W\left( (1+\epsilon) e^{i \theta}\right)\right]  \, d \theta\\
&=- \int_\alpha^\beta \left[ R_o \left( (1-\epsilon) e^{i \theta}\right) - R_{o,+}(e^{i \theta}) \right] \, d \theta - \int_\alpha^\beta \left[ R_{o,+}(e^{i \theta}) - R_o \left( (1+\epsilon) e^{i \theta}\right)   W\left( (1+\epsilon) e^{i \theta}\right)\right] \, d \theta\\
&= -\int_\alpha^\beta \left[ R_o \left( (1-\epsilon) e^{i \theta}\right) - R_{o,+}(e^{i \theta}) \right] \, d \theta - \int_\alpha^\beta \left[ R_{o,-}(e^{i \theta}) - R_o \left( (1+\epsilon) e^{i \theta}\right) \right] W\left( (1+\epsilon) e^{i \theta}\right) \, d \theta.
\end{align}
Taking absolute values, we have
\begin{align*}
\left| \int_{\gamma_{2}} R_{1}(s) ds + \int_{\gamma_4} R_{1}(s) ds \right| & \leq \int_\alpha^\beta \left| R_o \left( (1-\epsilon) e^{i \theta}\right) - R_{o,+}(e^{i \theta}) \right| \, d \theta \\& \quad + \int_\alpha^\beta \left| \left[ R_{o,-}(e^{i \theta}) - R_o \left( (1+\epsilon) e^{i \theta}\right) \right] W\left( (1+\epsilon) e^{i \theta}\right) \right| \, d \theta\\
& \leq \sqrt{\beta - \alpha} \norm{R_o - R_{o +}}_{2,(\alpha,\beta)} + \norm{R_{o,-} - R_o}_{2,(\alpha,\beta)} \, \norm{W}_{2,(\alpha,\beta)}
\end{align*}
where $ \norm{W}_{2,(\alpha,\beta)}$ is bounded, since the jump matrix $W$ is bounded away from $(-\zeta,0)$. Moreover, $\norm{R_o - R_{o +}}_{2,(\alpha,\beta)}$ and $\norm{R_{o,-} - R_o}_{2,(\alpha,\beta)}$ both approach zero as $\epsilon \to 0$ because of \eqref{e:l2_convergence}.  Next we estimate the integrals along the two blue line segments connecting the two red arcs.  Let $\gamma_{v} \coloneqq \gamma_1 \cup \gamma_3$ denote those two blue line segments, of length $2\epsilon$ connecting the two red arcs closest to the unit circle.  We have 
\begin{align*}
 \int_{\gamma_v} R_1(s) ds &=  \int_{1-\epsilon}^{1} R_o(t e^{i \alpha}) W(t e^{i \alpha}) \, dt + \int_{1}^{1+\epsilon} R_o(t e^{i \alpha}) \, dt + \int_{1+\epsilon}^{1} R_o(t e^{i \beta}) \, dt + \int_{1}^{1-\epsilon} R_o(t e^{i \beta}) W(t e^{i \beta}) \, dt.
\end{align*}
We claim that all integrals converge to zero as $\epsilon\to0$, we show this only for the second one, since the rest follow similarly. Consider
\begin{align*}
 \int_{1}^{1+\epsilon} R_o(t e^{i \alpha}) \, dt &= \int_0^{\epsilon} R_o \left( (u+1)e^{i \alpha}) \right) \, du\\
 &= \frac{1}{2 \pi i} \int_0^{\epsilon} \int_{\Sigma} \frac{\mu_o(s)(W(s)-I)}{s-(u+1) e^{i \alpha}} \, ds  \, du\\
 &= \frac{1}{2 \pi i} \int_{\Sigma} \mu_o(s)(W(s)-I) \left( \int_0^{\epsilon} \frac{1}{s-(u+1)e^{i \alpha}} \, du \right) \, ds\\
 &= - \frac{1}{2 \pi i e^{i \alpha}} \int_{\Sigma} \mu_o(s)(W(s)-I) \ln \left( \frac{s - (1+\epsilon) e^{i \alpha}}{s-e^{i \alpha}}\right)\,, \, ds
\end{align*}
we pass absolute values:
\begin{align*}
\left|  \int_{1}^{1+\epsilon} R_o(t e^{i \alpha}) \, dt  \right| & \leq \frac{1}{2 \pi} \int_{\Sigma} \left| \mu_o(s)(W(s)-I) \ln \left( \frac{s - (1+\epsilon) e^{i \alpha}}{s- e^{i \alpha}}\right) \right| \, ds\\
&= \frac{1}{2 \pi} \norm{\mu_o(W-I)}_{2,\Sigma} \, \left\| \ln \left( \frac{s - (1+\epsilon) e^{i \alpha}}{s-e^{i \alpha}}\right) \right\|_{2, \Sigma}
\end{align*}
where $\norm{\mu_o(W-I)}_{2,\Sigma}$ is bounded, and as $\epsilon \to 0$, $\left\| \ln \left( \frac{s - (1+\epsilon) e^{i \alpha}}{s-e^{i \alpha}}\right) \right\|_{2, \Sigma}$ approaches zero.  So we have shown that the original integral must vanish, and this implies that for any closed smooth contour $\gamma$ within a small neighborhood of $\Sigma$, $\oint_{\gamma}R_{1}(s)ds = 0$, which by Morera's Theorem proves that $R_{1}(Z)$ is analytic for $Z$ in a neighborhood of $\Sigma$.

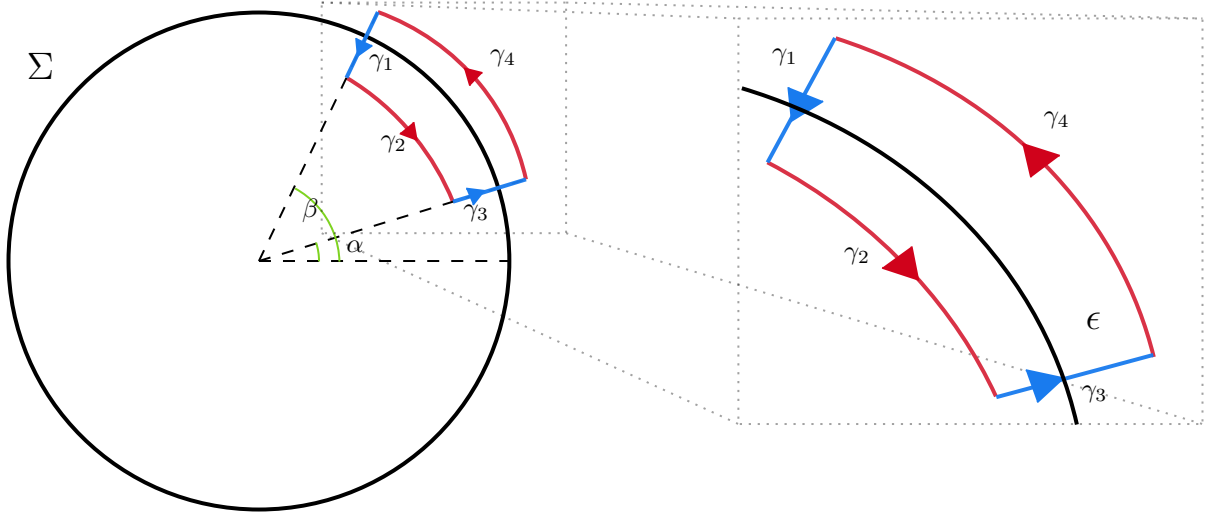
\begin{figure}
\centering

\tikzset{every picture/.style={line width=0.75pt}} 

\begin{tikzpicture}[x=0.75pt,y=0.75pt,yscale=-1,xscale=1]

\draw  [line width=1.5]  (34.67,149) .. controls (34.67,79.96) and (90.63,24) .. (159.67,24) .. controls (228.7,24) and (284.67,79.96) .. (284.67,149) .. controls (284.67,218.04) and (228.7,274) .. (159.67,274) .. controls (90.63,274) and (34.67,218.04) .. (34.67,149) -- cycle ;
\draw  [draw opacity=0][line width=1.5]  (203.37,56.8) .. controls (215.88,64.3) and (227.49,74.29) .. (237.38,86.61) .. controls (245.57,96.81) and (251.95,107.82) .. (256.52,119.17) -- (159.67,149) -- cycle ; \draw  [color={rgb, 255:red, 208; green, 2; blue, 27 }  ,draw opacity=0.81 ][line width=1.5]  (203.37,56.8) .. controls (215.88,64.3) and (227.49,74.29) .. (237.38,86.61) .. controls (245.57,96.81) and (251.95,107.82) .. (256.52,119.17) ;  
\draw  [draw opacity=0][line width=1.5]  (219.04,23.74) .. controls (238.75,30.99) and (256.45,43.21) .. (270.16,60.29) .. controls (281.62,74.56) and (289.17,90.82) .. (292.98,107.94) -- (159.67,149) -- cycle ; \draw  [color={rgb, 255:red, 208; green, 2; blue, 27 }  ,draw opacity=0.81 ][line width=1.5]  (219.04,23.74) .. controls (238.75,30.99) and (256.45,43.21) .. (270.16,60.29) .. controls (281.62,74.56) and (289.17,90.82) .. (292.98,107.94) ;  
\draw [line width=0.75]  [dash pattern={on 4.5pt off 4.5pt}]  (159.67,149) -- (219.04,23.74) ;
\draw [line width=0.75]  [dash pattern={on 4.5pt off 4.5pt}]  (159.67,149) -- (292.98,107.94) ;
\draw [color={rgb, 255:red, 25; green, 123; blue, 238 }  ,draw opacity=0.95 ][line width=1.5]    (203.37,56.8) -- (219.04,23.74) ;
\draw [color={rgb, 255:red, 25; green, 123; blue, 238 }  ,draw opacity=0.95 ][line width=1.5]    (256.52,119.17) -- (292.98,107.94) ;
\draw [line width=0.75]  [dash pattern={on 4.5pt off 4.5pt}]  (159.67,149) -- (284.67,149) ;
\draw  [draw opacity=0][line width=0.75]  (188.19,139.69) .. controls (189.22,142.85) and (189.7,146.08) .. (189.67,149.25) -- (159.67,149) -- cycle ; \draw  [color={rgb, 255:red, 126; green, 211; blue, 33 }  ,draw opacity=1 ][line width=0.75]  (188.19,139.69) .. controls (189.22,142.85) and (189.7,146.08) .. (189.67,149.25) ;  
\draw  [draw opacity=0][line width=0.75]  (177.73,112.51) .. controls (186.01,117.07) and (192.9,124.53) .. (196.75,134.17) .. controls (198.75,139.16) and (199.75,144.3) .. (199.86,149.34) -- (159.67,149) -- cycle ; \draw  [color={rgb, 255:red, 126; green, 211; blue, 33 }  ,draw opacity=1 ][line width=0.75]  (177.73,112.51) .. controls (186.01,117.07) and (192.9,124.53) .. (196.75,134.17) .. controls (198.75,139.16) and (199.75,144.3) .. (199.86,149.34) ;  
\draw  [color={rgb, 255:red, 208; green, 2; blue, 27 }  ,draw opacity=1 ][fill={rgb, 255:red, 208; green, 2; blue, 27 }  ,fill opacity=1 ] (236.82,79.12) -- (238.35,87.92) -- (229.99,84.78) ;
\draw  [color={rgb, 255:red, 208; green, 2; blue, 27 }  ,draw opacity=1 ][fill={rgb, 255:red, 208; green, 2; blue, 27 }  ,fill opacity=1 ] (265.08,60.82) -- (262.59,52.25) -- (271.24,54.44) ;
\draw  [color={rgb, 255:red, 25; green, 123; blue, 238 }  ,draw opacity=1 ][fill={rgb, 255:red, 25; green, 123; blue, 238 }  ,fill opacity=1 ] (216.27,40.51) -- (209.01,45.71) -- (208.21,36.82) ;
\draw  [color={rgb, 255:red, 25; green, 123; blue, 238 }  ,draw opacity=1 ][fill={rgb, 255:red, 25; green, 123; blue, 238 }  ,fill opacity=1 ] (263.09,112.58) -- (271.95,113.7) -- (266.5,120.77) ;
\draw  [color={rgb, 255:red, 0; green, 0; blue, 0 }  ,draw opacity=0.37 ][fill={rgb, 255:red, 0; green, 0; blue, 0 }  ,fill opacity=0 ][dash pattern={on 0.84pt off 2.51pt}] (191,19) -- (313,19) -- (313,135) -- (191,135) -- cycle ;
\draw [color={rgb, 255:red, 0; green, 0; blue, 0 }  ,draw opacity=0.37 ][fill={rgb, 255:red, 0; green, 0; blue, 0 }  ,fill opacity=0 ] [dash pattern={on 0.84pt off 2.51pt}]  (191,135) -- (399,231) ;
\draw [color={rgb, 255:red, 0; green, 0; blue, 0 }  ,draw opacity=0.37 ][fill={rgb, 255:red, 0; green, 0; blue, 0 }  ,fill opacity=0 ] [dash pattern={on 0.84pt off 2.51pt}]  (313,135) -- (630.67,231) ;
\draw  [color={rgb, 255:red, 0; green, 0; blue, 0 }  ,draw opacity=0.37 ][fill={rgb, 255:red, 0; green, 0; blue, 0 }  ,fill opacity=0 ][dash pattern={on 0.84pt off 2.51pt}] (399,27) -- (630.67,27) -- (630.67,231) -- (399,231) -- cycle ;
\draw [color={rgb, 255:red, 0; green, 0; blue, 0 }  ,draw opacity=0.37 ][fill={rgb, 255:red, 0; green, 0; blue, 0 }  ,fill opacity=0 ] [dash pattern={on 0.84pt off 2.51pt}]  (191,19) -- (399,27) ;
\draw [color={rgb, 255:red, 0; green, 0; blue, 0 }  ,draw opacity=0.37 ][fill={rgb, 255:red, 0; green, 0; blue, 0 }  ,fill opacity=0 ] [dash pattern={on 0.84pt off 2.51pt}]  (313,19) -- (630.67,27) ;
\draw  [draw opacity=0][line width=1.5]  (413.49,99.37) .. controls (440.41,113.58) and (465.37,132.5) .. (486.65,155.85) .. controls (504.13,175.05) and (517.76,195.75) .. (527.56,217.08) -- (319.67,273.99) -- cycle ; \draw  [color={rgb, 255:red, 208; green, 2; blue, 27 }  ,draw opacity=0.81 ][line width=1.5]  (413.49,99.37) .. controls (440.41,113.58) and (465.37,132.5) .. (486.65,155.85) .. controls (504.13,175.05) and (517.76,195.75) .. (527.56,217.08) ;  
\draw  [draw opacity=0][line width=1.5]  (447.57,36.93) .. controls (489.8,50.67) and (527.71,73.77) .. (557.08,106.02) .. controls (581.97,133.34) and (598.26,164.5) .. (606.36,197.29) -- (319.67,273.99) -- cycle ; \draw  [color={rgb, 255:red, 208; green, 2; blue, 27 }  ,draw opacity=0.81 ][line width=1.5]  (447.57,36.93) .. controls (489.8,50.67) and (527.71,73.77) .. (557.08,106.02) .. controls (581.97,133.34) and (598.26,164.5) .. (606.36,197.29) ;  
\draw [color={rgb, 255:red, 25; green, 123; blue, 238 }  ,draw opacity=0.95 ][line width=1.5]    (413.55,99.41) -- (447.22,36.81) ;
\draw [color={rgb, 255:red, 25; green, 123; blue, 238 }  ,draw opacity=0.95 ][line width=1.5]    (527.75,217.5) -- (606.09,196.24) ;
\draw  [color={rgb, 255:red, 208; green, 2; blue, 27 }  ,draw opacity=1 ][fill={rgb, 255:red, 208; green, 2; blue, 27 }  ,fill opacity=1 ] (485.43,141.67) -- (488.71,158.34) -- (470.75,152.39) ;
\draw  [color={rgb, 255:red, 208; green, 2; blue, 27 }  ,draw opacity=1 ][fill={rgb, 255:red, 208; green, 2; blue, 27 }  ,fill opacity=1 ] (545.13,106.91) -- (541.28,90.25) -- (559.42,96.09) ;
\draw  [color={rgb, 255:red, 25; green, 123; blue, 238 }  ,draw opacity=1 ][fill={rgb, 255:red, 25; green, 123; blue, 238 }  ,fill opacity=1 ] (441.29,68.57) -- (425.68,78.4) -- (423.95,61.56) ;
\draw  [color={rgb, 255:red, 25; green, 123; blue, 238 }  ,draw opacity=1 ][fill={rgb, 255:red, 25; green, 123; blue, 238 }  ,fill opacity=1 ] (542.44,203.9) -- (561.16,207.94) -- (548.15,220.08) ;
\draw  [draw opacity=0][line width=1.5]  (567.97,231.24) .. controls (558.87,190.76) and (537.06,151.96) .. (502.42,119.83) .. controls (473.19,92.71) and (438.17,73.45) .. (400.74,62.1) -- (319.67,273.99) -- cycle ; \draw  [line width=1.5]  (567.97,231.24) .. controls (558.87,190.76) and (537.06,151.96) .. (502.42,119.83) .. controls (473.19,92.71) and (438.17,73.45) .. (400.74,62.1) ;  

\draw (42.33,43.07) node [anchor=north west][inner sep=0.75pt]  [font=\Large]  {$\Sigma $};
\draw (201.86,145.94) node [anchor=south west] [inner sep=0.75pt]    {$\alpha $};
\draw (179.73,115.91) node [anchor=north west][inner sep=0.75pt]    {$\beta $};
\draw (268.2,119.23) node [anchor=north] [inner sep=0.75pt]    {$\gamma _{3}$};
\draw (273.67,52.37) node [anchor=south west] [inner sep=0.75pt]    {$\gamma _{4}$};
\draw (571,173.6) node [anchor=north west][inner sep=0.75pt]  [font=\Large]  {$\epsilon $};
\draw (234.92,84.61) node [anchor=north east] [inner sep=0.75pt]    {$\gamma _{2}$};
\draw (213.2,43.67) node [anchor=north west][inner sep=0.75pt]    {$\gamma _{1}$};
\draw (550,83.7) node [anchor=south west] [inner sep=0.75pt]    {$\gamma _{4}$};
\draw (414.33,40.84) node [anchor=north west][inner sep=0.75pt]    {$\gamma _{1}$};
\draw (450.33,140.84) node [anchor=north west][inner sep=0.75pt]    {$\gamma _{2}$};
\draw (568.92,210.27) node [anchor=north west][inner sep=0.75pt]    {$\gamma _{3}$};

\end{tikzpicture}

\caption{Closed curves $\gamma$, $\gamma_h$ and $\gamma_v$}
\label{fig:analyticity}
\end{figure}

To conclude our proof, we show that $R_o(z)$ is identical zero. First, \eqref{e:funRo} shows that:
\begin{equation}\label{e:asymRo}
R_o(Z) = -\frac{1}{2 \pi i} \int_{\Sigma} \mu_o(s) \left( W(s) - I \right) \sum_{j=0}^{\infty} \frac{s^j}{Z^{j+1}}  \, ds = \cO(Z^{-1}), \quad z \to \infty.
\end{equation}
Consider:
\[A(Z) \coloneqq R_o(Z) \, R_o^\dagger(\overline{Z})\,,\]
where $R_o^\dagger(Z)$ is the conjugate-transpose of $R(Z)$. We claim that $A(Z)$ is an entire function. Indeed, for $Z \in \Sigma$, we have: 
\begin{equation}
\label{eq:AA_inverse}
\begin{split}
    A_{+}(Z) &= \left( R_o(Z) \, R_o^\dagger(\overline{Z}) \right)_{+} = \left(R_o\right)_{+}(Z) \,\left(  \left(R_o\right)_{+}(\overline{Z})\right)^\dagger \\
    &=\left(R_o\right)_{-}W(Z) \left( \left(R_o\right)_{-}(\overline{Z}) W(\overline{Z})\right)^{\dagger} \\
    &= \left(R_o\right)_{-}W(Z)W(\overline{Z})^{\dagger}\left(R_o\right)_{-}(\overline{Z})^{\dagger} \\
    &= \left(R_o\right)_{-}\left(R_o\right)_{-}(\overline{Z})^{\dagger} \\
    &=A_{-}(Z),
\end{split}
\end{equation}
where one can easily verify that $W(Z) \, W^\dagger (\overline{Z})=I_2$. Thus $A$ is continuous everywhere in $\mathbb{C}$, and hence by Morera's theorem it is entire and vanishing at $\infty$, and so $A \equiv 0$. Combining \eqref{e:asymRo}-\eqref{eq:AA_inverse}, we conclude that $R_o(Z)$ is identically zero. Therefore, $\mu_o \equiv 0$, which implies that

\begin{equation}
    \Ker \left( \mathbbm{1} - C_{W} \right) = \{ 0 \}\,.
\end{equation}

The previous analysis allows us to apply standard Fredholm theory to show that $ \mathbbm{1} - C_{W}$ is a bijection. Indeed, since the operator $\mathbbm{1} - C_{W}$ is Fredholm, then:
\[\Ind(\mathbbm{1} - C_{W}) = \dim(\Ker \left( \mathbbm{1} - C_{W} \right)) - \dim(\Coker \left( \mathbbm{1} - C_{W} \right)).\]
We proved that  $\Ind(\mathbbm{1} - C_{W}) =0= \dim(\Ker \left( \mathbbm{1} - C_{W} \right))$, therefore:
\[\dim(\Coker \left( \mathbbm{1} - C_{W} \right))=0\]
which implies that $ \mathbbm{1} - C_{W}$ is a bijection. 
\end{proof}

\section{Proof of Theorem \ref{thm:relation}}\label{eq:appB}

In this section, we prove Theorem \ref{thm:relation}. Our first aim is to find an explicit relation between the normalization constants $c_j$ in RHP \ref{rhp:reflection_RHP} and the constants $p_j$ in the Darboux method. To establish this relation, we start with the following Lemma.
 
\begin{lemma}
\label{lem:inversion}
    Consider the scattering problem in the $x$-variable of Lax pair \eqref{eq:ZS1} with normalization 
    \begin{equation}
        \lim_{x\to+\infty}\Phi(x,z) N^{-1}(x,z)=I, \quad N^{-1}(x,z) = N^\dagger(x,\wo z )\,,
    \end{equation}
    where $N^\dagger$ is the conjugate transpose of $N$. 
Then:
    \begin{equation}
        \Phi^{-1}(x,z) = \wo{\Phi^\intercal(x,\wo z)} = \Phi^\dagger(x,\wo z).
    \end{equation}
\end{lemma}
\begin{remark}
    Typically (for Jost solutions) one uses $N(x,z) = e^{-i z x \sigma_{3}}$.  However, for the Darboux construction, one requires a more general (but still diagonal) matrix $N(x,z)$, see \eqref{eq:LargePhiAsymp} below.
\end{remark}
 \begin{proof}
     First, we write the initial value problem that $\Phi^\dagger(x,z)$ satisfies

     \begin{equation}
     \label{eq:dagger_system}
         \partial_{x} \Phi^\dagger(x,z) = \Phi^\dagger (x,z) \begin{pmatrix}
i\wo  z & -\psi\\
 \wo{\psi} & -i\wo z
\end{pmatrix}, \quad \lim_{x\to+\infty}\Phi^\dagger(x,z) \left(N^\dagger(x,z)\right)^{-1}=I \,.
     \end{equation}
     We can do the same for $\Phi^{-1}(x,z)$ obtaining
     \begin{equation}
         \partial_x \Phi^{-1}(x,z) =\Phi^{-1}(x,z) \begin{pmatrix}
i z & -\psi\\
 \wo{\psi} & -i z
\end{pmatrix}, \quad \lim_{x\to+\infty}\Phi^{-1}(x,z) N(x,z)=I\,.
     \end{equation}
The desired result is then obvious. 
 \end{proof}

Another relevant information to relate the two methods is the asymptotic behavior as $x \to\infty$ of the solution $\Phi_n(z;x,t)$; this is given in the following Lemma.
 \begin{lemma}
 \label{lem:gelash_asymp}
Consider the eigenfunctions $\Phi_n(x,z)$ constructed by the Darboux method (see \cite{gelash_strongly_2018} for more details). They satisfy the following asymptotic behavior:
\begin{equation}
\label{eq:LargePhiAsymp}
\begin{split}
\lim_{x\to - \infty}
    \Phi_n(x,z) 
    \Phi_{0}^{-1}(x,z) \pmtwo{\prod_{j=1}^n \frac{z - \lambda_j}{z - \overline{\lambda_j}}}{0}{0}{1}
    =I,  \\ \\[2pt]
\lim_{x\to + \infty}\Phi_n(x,z)  \Phi_{0}^{-1}(x,z) \pmtwo{1}{0}{0}{\prod_{j=1}^n \frac{z - \lambda_j}{z - \overline{\lambda_j}}}
    =I,  \\
\end{split}
\end{equation}
where $\Phi_0(x,z) = \begin{pmatrix}
    e^{-i z x} & 0 \\
    0 & e^{i z x}
\end{pmatrix}.$
 \end{lemma}

 \begin{proof}
From the Darboux method, we have:
\[\Phi_{n}(x,z) = \chi_{n}(x,z) \Phi_{n-1}(x,z)\]
where
\[\chi_{n}(x,z) = I_2 + \frac{\lambda_n - \overline{\lambda_n}}{z - \lambda_n} \frac{1}{|| q_n(x)||^2} \overline{q_n}(x) \cdot q_n^{T}(x), \quad q_{n}(x) = \overline{\Phi_{n-1}(x,\overline{\lambda_n})} \cdot \begin{pmatrix}
1\\
p_n
\end{pmatrix}\]
and $\cdot$ denotes inner product between vectors. First, we compute the asymptotic behavior of $q_n(x)$ as $x \to \pm \infty$ by induction. For this, we assume that there exists constant $c >0$ such that the relations
\begin{gather*}
\frac{q_j(x) \cdot q_j^{\dagger}(x)}{|| q_j(x)||^2} = \begin{pmatrix}
1 & 0\\
0 & 0
\end{pmatrix} + \mathcal{O}(e^{-c |x|}), \quad x \to - \infty \\ \\[2pt]
\frac{q_j(x) \cdot q_j^{\dagger}(x)}{|| q_j(x)||^2} = \begin{pmatrix}
0 & 0\\
0 & 1
\end{pmatrix} + \mathcal{O}(e^{-c |x|}), \quad x \to + \infty
\end{gather*}
hold for $j=1,2,\cdots,n-1$, and we show it for $j=n$. The proof is done for $x \to + \infty$, and the other limit is analogous. For $j=n$, we have:
\begin{align*}
q_{n}(x) & = \overline{\Phi_{n-1}(x,\overline{\lambda_n})} \cdot \begin{pmatrix}
1\\
p_n
\end{pmatrix} \\
& = \overline{\chi_{n-1}(x,\overline{\lambda_n}) \, \Phi_{n-2}(x,\lambda_n)} \cdot \begin{pmatrix}
1\\
p_n
\end{pmatrix} \\
& = \prod_{j=1}^{\overrightarrow{n-1}} \overline{\chi_{j}(x,\overline{\lambda_n})} \Phi_0(x,\overline{\lambda_n}) \cdot \begin{pmatrix}
1\\
p_n
\end{pmatrix} \\
& = \prod_{j=1}^{\overrightarrow{n-1}} \left( I_2 + \frac{\overline{\lambda_j} - \lambda_j}{\lambda_n - \overline{\lambda_j}} \frac{1}{|| q_j(x)||^2} q_j(x) \cdot q_j^{\dagger}(x) \right) \,  \Phi_0(x,\overline{\lambda_n}) \cdot \begin{pmatrix}
1\\
p_n
\end{pmatrix}\\
& = \prod_{j=1}^{\overrightarrow{n-1}} \left( I_2 +  \frac{\overline{\lambda_j} - \lambda_j}{\lambda_n - \overline{\lambda_j}} \left( \begin{pmatrix}
0 & 0\\
0 & 1
\end{pmatrix} + \mathcal{O}(e^{-c |x|}) \right) \right)\Phi_0(x,\overline{\lambda_n}) \cdot \begin{pmatrix}
1\\
p_n
\end{pmatrix}
\end{align*}
which entry-wise gives
\begin{gather*}
q_{n1}(x) = e^{i \lambda_n x} \left( 1 + \mathcal{O}(e^{-c |x|}) \right), \quad q_{n2}(x) = e^{-i \lambda_n x} p_n \left( \prod_{j=1}^{n-1} \frac{\wo{\lambda_j} - \lambda_j}{\lambda_n - \overline{\lambda_j}} + \mathcal{O}(e^{-c |x|})\right), \quad \quad x \to + \infty.
\end{gather*}
The last equality implies that
\[|| q_{n}(x)||^2 \coloneqq |q_{n1}(x)|^2 + |q_{n2}(x)|^2 = e^{i(\lambda_n - \overline{\lambda_n})x}\left( 1 + \mathcal{O}(e^{-c |x|}) \right) + e^{-i(\lambda_n - \overline{\lambda_n})x}|p_n|^2 \left( \left\vert \prod_{j=1}^{n-1} \frac{\wo{\lambda_j} - \lambda_j}{\lambda_n - \overline{\lambda_j}}\right\vert^2 + \mathcal{O}(e^{-c |x|}) \right) \]
or alternatively
\[|| q_{n}(x)||^2 \sim |q_{n2}(x)|^2, \quad x \to + \infty\]
which implies that for $j=n$ we get the desired asymptotics:
\[\frac{q_n(x) \cdot q_n^{\dagger}(x)}{|| q_n(x)||^2} = \begin{pmatrix}
0 & 0\\
0 & 1
\end{pmatrix} + \mathcal{O}(e^{-c |x|}), \quad x \to + \infty\]
which concludes the induction, and therefore the proof.
 \end{proof}
Finally, we can prove Theorem \ref{thm:relation}, for the sake of the reader, we rewrite the statement
 \begin{theorem}
 \label{thm:relation_app}
     Consider the spectral data $D_N(t)$ associated to RHP \ref{rhp:reflection_RHP} and $\mathfrak{D}_N(t)$ 
     associated to the Darboux method, defined in \eqref{eq:spectral_RHP}-\eqref{eq:spectral_DM}, respectively. Suppose they are related via:
     \begin{equation}
     \label{eq:relation_app}
         c_n= \frac{1}{p_n}\prod_{\ell=1}^N (\lambda_n -\wo{\lambda_\ell})\prod_{\genfrac{}{}{0pt}{}{\ell=1}{\ell\neq n}}^N\frac{1}{\lambda_n -\lambda_\ell}\,,
     \end{equation}
     then the NLS solution arising from the two methods is the same. 
 \end{theorem}
\begin{remark}
This result implies that $\gamma_n= \frac{1}{p_n}$, where $\gamma_n$ are the proportionality constants of the classical inverse scattering theory (see for example \cite{NLS_book}). 
\end{remark}
 \begin{proof}
  Define $a_N(z) = \prod_{n=1}^N \frac{z-\lambda_n}{z-\wo{\lambda_n}}$. First, we want to show that for all $\ell=1,\ldots, N$

    \begin{equation}
    \label{eq:claim_1}
        \lim_{z\to \lambda_\ell} a_N(z)\Phi^{(1)}_N(z) = \frac{1}{p_\ell} \lim_{z\to \lambda_\ell} a_N(z)\Phi^{(2)}_N(z)
    \end{equation}
    where $\Phi^{(1)}_N(z),\Phi^{(2)}_N(z)$ are the first and the second column of $\Phi_N(z)$ in the Darboux method, respectively. We consider just the case $\ell=N$, the others are completely analogues. From their definition, one gets the following chain of equalities:

    \begin{equation}
    \begin{split}
                \Phi_N(z) & = \chi_N(z)\Phi_{N-1}(z) = \left( I_2 + \frac{\lambda_N - \wo{\lambda_N}}{(z-\lambda_N)||q_N||^2}\wo q_N q_N^\intercal\right)\Phi_{N-1}(z) \\
                & = \left( I_2 + \frac{\lambda_N - \wo{\lambda_N}}{(z-\lambda_N)||q_N||^2}\Phi_{N-1}(\wo{\lambda_N})\begin{pmatrix}
                    1\\\wo p_N
                \end{pmatrix}\begin{pmatrix}
                    1&  p_N
                \end{pmatrix}\Phi^\dagger_{N-1}(\wo{\lambda_N})\right)\Phi_{N-1}(z) \,.
    \end{split}
    \end{equation}
Therefore, when evaluating 

\begin{equation}
\begin{split}
     \lim_{z\to \lambda_N} a_N(z)\Phi_N(z) & = a'(\lambda_N)(\lambda_N-\wo{\lambda_N})\Phi_{N-1}(\wo{\lambda_N})\begin{pmatrix}
        1 & p_N\\
        \wo{p_N} & |p_N|^2
    \end{pmatrix}\Phi^\dagger_{N-1}(\wo{\lambda_N})\Phi_{N-1}(\wo{\lambda_N}) \\[1pt] & = a'(\lambda_N)(\lambda_N-\wo{\lambda_N})\Phi_{N-1}(\wo{\lambda_N})\begin{pmatrix}
        1 & p_N\\
        \wo{p_N} & |p_N|^2
    \end{pmatrix}
\end{split}
\end{equation}
where in the last equality we used Lemma \ref{lem:inversion}. Moreover, defining  $U = a'(\lambda_N)(\lambda_N-\wo{\lambda_N})\Phi_{N-1}(\wo{\lambda_N})$ 
\begin{equation}
    U \begin{pmatrix}
        1 & p_N\\
        \wo{p_N} & |p_N|^2
    \end{pmatrix} = \begin{pmatrix}
    u_{11} & u_{12} \\
    u_{21} & u_{22}
\end{pmatrix}\begin{pmatrix}
        1 & p_N\\
        \wo{p_N} & |p_N|^2
    \end{pmatrix} =  \begin{pmatrix}
        u_{11} + \wo{p_N} u_{21} & p_N(u_{11} + \wo{p_N} u_{21}) \\
        u_{21} + \wo{p_N} u_{22} & p_N(u_{21} + \wo{p_N} u_{22}) \\
    \end{pmatrix}\,,
\end{equation}
i.e., the first column equals the second column times $1/p_N$, which proves \eqref{eq:claim_1}. Now, we claim that the matrix $m(z)$ defined as follows:
\begin{equation}
\label{eq:sol_rhp}
    m(z) = \begin{pmatrix}
        \Phi_N^{(1)}(z)e^{izx} & a_N(z)\Phi_N^{(2)}(z)e^{-izx}
    \end{pmatrix}
\end{equation}
solves RHP \ref{rhp:reflection_RHP} with $r(z)\equiv 0$. Indeed, given Lemma \ref{lem:gelash_asymp},  the first columns of $m(z)$ is the Jost solution normalized as $e^{-izx}$ at $x\to,+\infty$, multiplied by $\frac{e^{izx}}{a_N(z)}$, while the second column is the Jost solution normalized as $e^{izx}$ at $x\to,+\infty$, multiplied by $e^{-izx}$. Given the construction in \cite[Chapter 2]{NLS_book}, one immediately deduces this claim. 

We can now prove \eqref{eq:relation_app}. The idea is to consider \eqref{eq:sol_rhp}, compute the residue of this matrix at $\lambda_N$ and get an expression in terms of the spectral data of the RHP and of the Darboux method. Using the definition of $m$:
\begin{equation}
\begin{split}
    \Res_{\lambda_N} m(z) &= \lim\limits_{z\to \lambda_N}(z-\lambda_N)
        \begin{pmatrix}
        \Phi_N^{(1)}(z)e^{izx} & a_N(z)\Phi_N^{(2)}(z)e^{-izx}
    \end{pmatrix}=
        \begin{pmatrix}
        \lim\limits_{z\to \lambda_N}(z-\lambda_N)\frac{a_N'(z_N)}{a_N'(z_N)}\Phi_N^{(1)}(z)e^{izx} & 0
    \end{pmatrix}\\
    &=
        \begin{pmatrix}
        \lim\limits_{z\to \lambda_N}\frac{a_N(z)}{a_N'(\lambda_N)}\Phi_N^{(1)}(z)e^{izx} & 0
    \end{pmatrix}\stackrel{\eqref{eq:claim_1}}{=}
        \begin{pmatrix}
        \lim\limits_{z\to \lambda_N}\frac{a_N(z)}{p_N a'(\lambda_N)}\Phi_N^{(2)}(z)e^{izx} & 0
    \end{pmatrix}\\
    & = \begin{pmatrix}
        \frac{e^{i \lambda_N x}}{p_N a'(\lambda_N)}\lim\limits_{z\to \lambda_N}a_N(z)\Phi_N^{(2)}(z) & 0
    \end{pmatrix}.\\
\end{split}
\end{equation}
We notice that  $\lim_{z\to \lambda_N}a_N(z)\Phi_N^{(2)}(z)$ is finite and different from zero. Under the RHP perspective, at $t=0$, in view of \eqref{eq:rescondCP}
\begin{equation}
        \Res_{\lambda_N} m(z) = \lim_{z\to \lambda_N} m(z)\begin{pmatrix}
			0 & 0 \\
			c_N e^{2i\theta(\lambda_N;x,0)} & 0
		\end{pmatrix}=\begin{pmatrix}
		    c_N e^{i \lambda_N x} \lim\limits_{z \to \lambda_N} a_N(z)\Phi_N^{(2)}(z)& 0
		\end{pmatrix}\,.
\end{equation}
The previous equations imply equation \eqref{eq:relation_app} as desired. 
 \end{proof}

 \section{Proof of Lemma \ref{lem:last_boredoom}}
\label{app:epsilon_net}
 In this section we prove Lemma \ref{lem:last_boredoom} using the so called $\epsilon$-net argument. For the sake of the reader, we report the statement of the lemma

 \begin{lemma}
   Fix $\zeta\in (0,1)$, and $\frac{1}{4}< \delta$. Let $X_1,\ldots, X_N$ be positive i.i.d. sub-exponential random variable with mean $\mu$. Then, there exists an $N_0>0$ and two positive constants $c,\gamma>0$ independent of $N$ such that  

   \begin{equation}
       \mathbb{P}(\cU^c)\leq e^{-cN^\gamma}\,,
   \end{equation}
   where
   \begin{equation}
       \cU =\left\{ (X_1,\ldots,X_N) \in \R_+^N \Big\vert\,\, \forall Z: |Z|=1, \Bigg|\sum_{j=1}^N \frac{  X_j - \mu }{Z+\zeta \frac{j}{N}}\Bigg|\leq N^{2 \, \delta} \right\}\,.
   \end{equation}
\end{lemma}

\begin{proof}
    We notice that 

    \begin{equation}
        \cU^c =\left\{ (X_1,\ldots,X_N) \in \R_+^N \Big\vert\,\, \exists\, Z: |Z|=1, \Bigg|\sum_{j=1}^N \frac{  X_j - \mu }{Z+\zeta \frac{j}{N}}\Bigg|> N^{2 \, \delta} \right\}\,.
    \end{equation}
    Define the following points and sets 
    \begin{equation}
        Z_k = e^{2\pi i\frac{k}{N}}\,, \qquad \cB_k = \left\{Z\in\C \; \Big\vert |Z|=1, \, 2\pi \frac{2k-1}{2N}\leq\arg\left(Z\right)< 2\pi \frac{2k+1}{2N}\right\}\,,
    \end{equation}
    then the following holds

    \begin{equation}
    \label{eq:split_net}
        \mathbb{P}(\cU^c) \leq \sum_{k=1}^N \mathbb{P}\left(\left\{ (X_1,\ldots,X_N) \in \R_+^N \Big\vert\,\, \exists\, Z: Z\in\cB_k, \Bigg|\sum_{j=1}^N \frac{  X_j - \mu }{Z+\zeta \frac{j}{N}}\Bigg|> N^{2 \, \delta} \right\}\right)\,.
    \end{equation}
    Since $0<\zeta<1$, there exists a constant $\nu>0$ such that 

    \begin{equation}
        \left\vert Z+\zeta \frac{j}{N} \right\vert > \nu \quad \forall \;Z \,:\; |Z|=1\,,
    \end{equation}
    defining $S(Z) = \sum_{j=1}^N \frac{  X_j - \mu }{Z+\zeta \frac{j}{N}}$, using  Lipschitz inequality, one can show that for all $Z\in\cB_k$

    \begin{equation}
        \left| S(Z) - S(Z_k)\right| \leq |Z-Z_k| \sup_{Z\in \cB_k} |S'(Z)|\leq 2\frac{\sum_{j=1}^N |X_j-\mu|}{N\nu^2} \leq 2\frac{\max_{j=1,\ldots,N}(|X_j-\mu|)}{\nu^2}\,.
    \end{equation}
    Finally, by the previous proof and the Bernstein inequality (adapted to the of complex coefficients) for $|S(Z_k)|$ we conclude that

    \begin{equation}
    \begin{split}
              \mathbb{P}\Bigg(\Bigg\{ (X_1,\ldots,X_N) \in \R_+^N \Big\vert\,&\, \exists\, Z: \in\cB_k, \Bigg|\sum_{j=1}^N \frac{  X_j - \mu }{Z+\zeta \frac{j}{N}}\Bigg|> N^{2 \, \delta} \Bigg\}\Bigg) \leq  \mathbb{P}\left(\left\{ (X_1,\ldots,X_N) \in \R_+^N \Big\vert\,\, |S(Z_k)|> \frac{N^{2 \, \delta}}{2}\right\} \right)  \\ &\qquad  + \mathbb{P}\left(\left\{ (X_1,\ldots,X_N) \in \R_+^N \Big\vert\,\, \exists\, Z: |Z|\in \cB_k\,, \Bigg|S(Z_k) - S(Z)\Bigg|>\frac{ N^{2 \, \delta}}{2}\right\} \right) \\
              & \leq \frac{e^{-cN^{2\varepsilon}}}{2} + \mathbb{P}\left(\left\{ (X_1,\ldots,X_N) \in \R_+^N \Big\vert\,\, \max(|X_j-\mu|)>\frac{\nu N^{2 \, \delta}}{4}\right\}\right)\\
              & \leq e^{-cN^{2\varepsilon}}
    \end{split}
    \end{equation}
    where in the last inequality can be deduced by arguing as in Lemma \ref{lem:scaling_max}. Therefore, by using the previous estimate in \eqref{eq:split_net} we conclude.
\end{proof}

\bibliographystyle{siam}
\bibliography{biblio}

\end{document}